\documentclass[a4paper,10pt,notitlepage,reqno]{article}
\usepackage[english]{babel}
\usepackage{amssymb,amsthm,amsmath} 
\usepackage{mathtools}
\usepackage{mathrsfs}
\usepackage{enumerate}
\usepackage{authblk}
\usepackage{cite}
\usepackage[symbol]{footmisc}
\usepackage{color}
\usepackage{geometry}
\usepackage[colorlinks=true, linkcolor=red, citecolor=blue, urlcolor=blue,hyperfootnotes=false]{hyperref}
\usepackage{cleveref}
\theoremstyle{plain} 
\newtheorem{theorem}{Theorem}[section]
\newtheorem{lemma}{Lemma}[section]

\newtheorem{corollary}{Corollary}[section]

\theoremstyle{definition}
\newtheorem{definition}{Definition}[section]

\theoremstyle{remark}
\newtheorem{remark}{Remark}[section]

\DeclareMathOperator{\curl}{curl}

\DeclareMathOperator{\dist}{dist}

\DeclareMathOperator{\ess}{ess}
\DeclareMathOperator{\Dom}{dom}

\newcommand{\R}{\mathbb{R}}
\newcommand{\Z}{\mathbb{Z}}
\newcommand{\C}{\mathbb{C}}

\newcommand{\N}{\mathbb{N}}
\newcommand{\normeq}[1]{{\left\vert\kern-0.25ex\left\vert\kern-0.25ex\left\vert #1 
    \right\vert\kern-0.25ex\right\vert\kern-0.25ex\right\vert}}

\newenvironment{system}%
{\left\lbrace\begin{array}{r@{\hspace{1mm}}ll}}%
{\end{array}\right.}

\title{\textbf{Improved Hardy-Rellich inequalities}}

\author[1]{Biagio Cassano}
\author[2]{Lucrezia Cossetti} 
\author[3]{Luca Fanelli}		
	
\affil[1]{Dipartimento di Matematica, Universit\`a degli
  Studi di Bari ``A.~Moro'', via Orabona 4, 70125 Bari, Italy.

  Dipartimento di Matematica e Fisica, Università degli Studi
  della Campania  ``Luigi Vanvitelli'', Viale Lincoln 5, 81100 Caserta, Italy; biagio.cassano@unicampania.it}

\affil[2]{Fakult\"{a}t f\"{u}r Mathematik, Institut f\"{u}r Analysis, Karlsruher Institut f\"{u}r Technologie (KIT), Englerstra{\ss}e 2, 76131 Karlsruhe, Germany; lucrezia.cossetti@kit.edu}

\affil[3]{Departamento de Matematicas, Universidad del Pa\'is Vasco/Euskal Herriko Unibertsitatea (UPV/EHU), Aptdo. 644, 48080, Bilbao, Spain; luca.fanelli@ehu.es}

\begin{document}

\date{\small 18 November 2021}
%%%%%%%%%%%%%%%%%%%%%%%%%%%%%%%%%%%%%%%%%%%%%%%%%%%%%%%%%%%%%%%%%%%%%%%%%%%%%%%%%%%%%%%%%%%%%%%%%%%%%%%%%%%%%%%%%%%%%%%%%%%%%%%%%%%%%%%%%%

\maketitle
%\vspace{-1cm}

%\nocite{*}

%------------abstract------------%

\begin{abstract}
	\noindent
We investigate Hardy-Rellich inequalities for perturbed
Laplacians. In particular, we show that a non-trivial angular
perturbation of the free operator typically improves the inequality,
and may also provide an estimate which does not hold in the free
case. The main examples are related to the introduction of a magnetic
field: this is a manifestation of the diamagnetic phenomenon, which
has been observed by Laptev and Weidl in \cite{LW1999} for the Hardy
inequality, later by Evans and Lewis in \cite{EL2005} for the Rellich
inequality; however, to the best of our knowledge, the so called Hardy-Rellich inequality has not yet 
been investigated in this regards. After showing the optimal inequality, we prove that the best constant is not attained by any function in the domain of the estimate.
%The proof of our main result is inspired by \cite{Cazacu2019}.
\end{abstract}



%---------------------------------%
\section{Introduction}
The Hardy inequality, introduced in \cite{Hardy20}, is one of the well known mathematical manifestations of the Uncertainty Principle in Quantum Mechanics. It affirms that
\begin{equation}\label{eq:Hardy-classical}
\int_{\R^d} |\nabla\psi(x)|^2\,dx
\geq \frac{(d-2)^2}{4} 
\int_{\R^d}\frac{|\psi(x)|^2}{|x|^2}\, dx,
\end{equation}
for any $\psi\in H^1(\R^d)$\footnote{$H^1(\R^d)$ denotes the Sobolev space of $L^2(\R^d)$ functions with first weak derivatives in $L^2(\R^d).$}, with $d\geq3.$ 
The low dimensions $d=1,2$ are not included in \eqref{eq:Hardy-classical}, 
since the weight $|x|^{-2}$ is not locally integrable. On the other hand, if $d=1$, the inequality \eqref{eq:Hardy-classical} holds for any function $\psi$ in the smaller domain $H^1(\R\setminus\{0\}):=\overline{C^\infty_0(\R\setminus\{0\})}^{\|\cdot\|_{H^1(\R)}}\subsetneq H^1(\R)$. The dimension $d=2$ is critical for the validity of \eqref{eq:Hardy-classical}, which cannot hold in this case with a non-zero constant on the right-hand side.

\medskip
\noindent
An analogous and more singular example is given by the Rellich inequality, introduced in \cite{Rellich53} (see also~\cite{RB1969}), which states that
\begin{equation}\label{eq:Rellich-classical}
	\int_{\R^d} |\Delta \psi(x)|^2\, dx
	\geq \frac{d^2 (d-4)^2}{16} 
	\int_{\R^d} \frac{|\psi(x)|^2}{|x|^4}\,dx,
\end{equation}
for any $\psi\in H^2(\R^d)$, with $d\geq5$, or $\psi\in H^2(\R^d\setminus\{0\})$, with $d=1,3$. The dimensions $d=2,4$ play for \eqref{eq:Rellich-classical} the same role of criticality as $d=2$ for the Hardy inequality. Both the constants on the right-hand sides of \eqref{eq:Hardy-classical}, \eqref{eq:Rellich-classical} are sharp, and not attained on any function in the corresponding domains.

\medskip
\noindent
The inequalities \eqref{eq:Hardy-classical}, \eqref{eq:Rellich-classical} are fundamental tools in order to describe scaling-critical perturbations of the free Hamiltonians $-\Delta$ and $\Delta^2$, respectively, by the standard perturbation theory for quadratic forms. In addition, they naturally come into play in a multitude of areas of Mathematics and Physics (elliptic PDEs with singular potentials, stability of quantum systems etc...). Due to their applications, these inequalities have both been objects of intense study (see~\cite{CK2016,BLS2004,FKLV,FS2008,KPP2018,Ozawa
%,CFKP2021
}  and~\cite{DH1998,Bennett1989,Yafaev1999}, respectively, and references therein, to cite a necessarily incomplete bibliographical list).

\medskip
\noindent
In this paper we are interested in the Hardy-Rellich inequality, which is in between \eqref{eq:Hardy-classical} and \eqref{eq:Rellich-classical}: 
\begin{equation}\label{eq:HardyRellich-classical}
	\int_{\R^d} |\Delta \psi(x)|^2\, dx
	\geq C(d)\int_{\R^d}\frac{|\nabla \psi(x)|^2}{|x|^2}\, dx, 
\end{equation}
for any $\psi\in H^2(\R^d)$ with $d\geq3,$ or $\psi\in H^2(\R\setminus\{0\})$ in the case $d=1$, where the constant $C(d)$ is given by
\begin{equation}\label{eq:cididdi}
C(d) =
\begin{cases}
\frac 14 & \text{if }d=1
\\
\frac{25}{36} & \text{if }d=3
\\
3 & \text{if }d=4
\\
\frac{d^2}{4} & \text{if }d\geq5.
\end{cases}
\end{equation}
The dimension $d=2$ is critical for the validity of \eqref{eq:HardyRellich-classical}, in the same way as for the previous inequalities.
\noindent
Similarly to \eqref{eq:Hardy-classical} and \eqref{eq:Rellich-classical}, inequality \eqref{eq:HardyRellich-classical} is useful to show the boundedness from below of either the biharmonic operator with second order perturbations (in the form sense) or the Laplacian with first order perturbations, via the Kato-Rellich Theorem (in the operator sense). Moreover, due to the trivial identity $\int_{\R^d}|\widehat{\psi}(\xi)|^2|\xi|^\beta\, d\xi=(2\pi)^{-\beta}\int_{\R^d}|(-\Delta)^{\beta/4}\psi(x)|^2\, dx,$ \eqref{eq:HardyRellich-classical} can be recast in the framework of Pitt's inequalities with gradient terms (see~\cite[Theorem 4]{Beckner2008}) which, in their classical formulation, are weighted inequalities involving a function and its Fourier transform and therefore intimately connected to quantifying uncertainty principles. Finally, \eqref{eq:HardyRellich-classical} serves as a tool to get improvement over more standard Rellich-type inequalities on bounded domains (see \cite{TZ2007}).

\medskip
\noindent
Surprisingly, despite being intimately linked to~\eqref{eq:Hardy-classical} and~\eqref{eq:Rellich-classical}, inequality~\eqref{eq:HardyRellich-classical}   
appeared for the first time much later than the former. 
In 2007 Tertikas and Zographopoulos~\cite{TZ2007} proved~\eqref{eq:HardyRellich-classical} for $d\geq 5$.
%Their method uses a spherical harmonics decomposition but their proof fails in lower dimensions. 
The lower dimensional cases $d=3,4$ were covered later independently by Beckner in~\cite{Beckner2008} 
%via the Fourier Transform
%have been used to prove very general Pitt's inequalities involving also gradient type terms 
and by Ghoussoub and Moradifam in~\cite{GM2011}.
%reducing the problem to constructing Bessel pairs of weights. 
%where introducing a characterization of such inequalities in terms of Bessel pairs of weights, were able to develop a powerful unified machinery which allows them to get improvements and generalizations of many Hardy/Rellich-type inequalities. 
Furthermore both these works recovered the higher dimensional case $d\geq 5$ already proved in~\cite{TZ2007}.
The method used in~\cite{GM2011} is reminiscent of the one used
in~\cite{TZ2007} for $d\geq 5$ and it is based on spherical harmonics
decomposition; however, the proof requires distinguishing between the lower and the higher dimensional setting. A compact and unified proof of~\eqref{eq:HardyRellich-classical} in \emph{any} dimension $d\geq 3,$ with optimal constants $C(3)=25/36,$ $C(4)=3$ and $C(d)=d^2/4$ if $d\geq 5,$ was recently obtained by Cazacu in~\cite{Cazacu2019}. He showed that the same technique applied in~\cite{TZ2007} to prove~\eqref{eq:HardyRellich-classical} for $d\geq 5$ could be extended (introducing an additional optimizing parameter) to cover any dimension $d\geq 3.$ In addition, the author showed the non-attainability of the best constant $C(d)$ for any $d\geq 3$ and he also provided minimizing sequences for $C(d)$ in lower dimensions $d=3,4$ (minimizing sequences in $d\geq 5$ were already constructed in~\cite{TZ2007}). Improvements of these inequalities on bounded domains can be found in~\cite{GM2011,Lam2018,NLN2019}.
Hardy-Rellich inequalities valid on Riemaniann manifolds are investigated in~\cite{KO2009,Nguyen2020}. Further generalizations can be found in~\cite{GL2018,Costa2009}.
To the best of our knowledge, the case $d=1$ is not written, anyway this is an immediate consequence of the classical 1D Hardy inequality. More precisely,~\eqref{eq:HardyRellich-classical} holds true in $d=1$ with $C(1)=1/4.$

\medskip
\noindent
The best constants of the above inequalities need to be understood as ground energy levels of suitable Hamiltonians. 
It is convenient to get a deeper insight to \eqref{eq:Hardy-classical} first, introducing the spherical coordinates in $\R^d$, $d\geq2$ to write the free Hamiltonian as
$$
-\Delta = -\frac{\partial^2}{\partial r^2}-\frac{d-1}{r}\frac{\partial}{\partial r}-\frac1{r^2}\Delta_{\mathbb S^{d-1}},
$$
where $-\Delta_{\mathbb S^{d-1}}$ is the Laplace-Beltrami operator on the unit sphere.  The spectrum of $-\Delta_{\mathbb S^{d-1}}$ is purely discrete, and it is given by the sequence $\sigma(-\Delta_{\mathbb S^{d-1}})=\{k(k+d-2)\}_{k =0,1,\dots}$. Then, if we rewrite \eqref{eq:Hardy-classical} using the language of quadratic forms, i.e.
\begin{equation}\label{eq:hardy2quadratic}
-\Delta \geq \frac{(d-2)^2}{4|x|^2},
\end{equation}
the fact that the lowest eigenvalue of $-\Delta_{\mathbb S^{d-1}}$ is 0 shows that the contribution to \eqref{eq:hardy2quadratic} entirely comes from the positive radial operator $L_r=-\frac{\partial^2}{\partial r^2}-\frac{d-1}{r}\frac{\partial}{\partial r}$. Therefore the following two facts are evident:
\begin{itemize}
\item[(i)] if one restricts to $L^2$--functions which are orthogonal to the eigenspace associated to the first eigenvalue of $-\Delta_{\mathbb S^{d-1}}$, then there is an improvement of the best constant in \eqref{eq:hardy2quadratic};
\item[(ii)] any angular perturbation to the operator $-\Delta_{\mathbb S^{d-1}}$ which increases the bottom of the spectrum gives a consequent improvement to the best constant in \eqref{eq:hardy2quadratic}.
\end{itemize}
A trivial example concerning fact (ii) above is obtained by fixing $a>0$ and considering the scaling invariant operator
$$
-\Delta + \frac{a}{|x|^2}= -\frac{\partial^2}{\partial
  r^2}-\frac{d-1}{r}\frac{\partial}{\partial r} + \frac1{r^2}\left(-\Delta_{\mathbb S^{d-1}}+a\right).
$$
Since $\sigma(-\Delta_{\mathbb S^{d-1}}+a)=\{k(k+d-2)+a\}_{k =0,1,\dots}$, we have the obvious inequality
\begin{equation}\label{eq:hardy3quadratic}
-\Delta+ \frac{a}{|x|^2} \geq \left(\frac{(d-2)^2}{4}+a\right)\frac1{|x|^2}.
\end{equation}
A completely analogous more general result can be easily obtained if $a$ is replaced by a $0$--degree homogeneous function $a(\theta):\mathbb S^{d-1}\to\R$, assuming that $\inf_{\mathbb S^{d-1}}a(\theta)=: a>0$.

\medskip
\noindent
A more geometric improvement occurs in presence of an external magnetic field. A magnetic Schr\"odinger Hamiltonian is an operator of the form $-\Delta_A = (-i\nabla + A)^2$, where $A:\R^d\to\R^d$, $d\geq2$.  The diamagnetic inequality 
$$
\left|(-i\nabla+A)\psi(x)\right|\geq|\nabla|\psi|(x)|,
\qquad
\text{for a.e. }x\in\R^d
$$
valid for $A\in L^2_{\text{loc}}$ (see e.g. \cite{Li_Lo}), together with \eqref{eq:hardy2quadratic}, immediately shows that
\begin{equation}\label{eq:hardy4quadratic}
-\Delta_A \geq \frac{(d-2)^2}{4|x|^2},
\end{equation}
for any vector potential $A\in L^2_{\text{loc}}(\R^d)$. In order to understand the role of $A$ in \eqref{eq:hardy4quadratic}, it is again convenient to describe a scaling invariant example. Let $A\in L^2_{\text{loc}}(\R^d)$ be of the form
\begin{equation}\label{eq:magnex}
A(x)=|x|^{-1}\mathbf{A}(\theta),
\qquad
\theta:=\frac{x}{|x|}
\end{equation} 
for some $0$--degree homogeneous vector field $\mathbf{A}:\mathbb S^{d-1}\to\mathbb S^{d-1}$. In addition, assume that $A$ is in the {\it transversal gauge} (or Cr\"onstrom, or Poincar\'e gauge, see \cite{I}), namely $x\cdot A(x)\equiv0$ for almost every $x\in\R^d$. Then the operator $-\Delta_A$ in spherical coordinates reads as
$$
-\Delta_A =-\frac{\partial^2}{\partial r^2}-\frac{d-1}{r}\frac{\partial}{\partial r}-\frac1{r^2}\Delta_{\mathbf A,\mathbb S^{d-1}},
$$
where $-\Delta_{\mathbf A,\mathbb S^{d-1}}=(-i\nabla_{\mathbb S^{d-1}}+\mathbf{A})^2$. As in the previous examples, the main contribution to the improvement in \eqref{eq:hardy4quadratic} comes from the fact that the bottom of the spectrum of $-\Delta_{\mathbf A,\mathbb S^{d-1}}$ is always non-negative, due to the spherical version of the diamagnetic inequality (see e.g. \cite{FFT2011}). Therefore it is natural to look for explicit examples of potentials such that $\min\sigma(-\Delta_{\mathbf A,\mathbb S^{d-1}})=a>0$, with a consequent quantitative improvement in \eqref{eq:hardy4quadratic}. The first example in this direction, at our knowledge, is due to Laptev and Weidl \cite{LW1999}. They proved in the two-dimensional case $d=2$ that
\begin{equation}\label{eq:LW}
	\int_{\R^2} |\nabla_{\!A} \psi(x)|^2\, dx\geq \dist\{\widetilde \Psi,\Z\}^2 \int_{\R^2} \frac{|\psi(x)|^2}{|x|^2}\, dx,
\end{equation}
for any $\psi\in H^1_A:=\left\{f\in L^2(\R^2):\int_{\R^2}|\nabla_A f|^2<\infty\right\}$, where $A$ is the {\it Aharonov-Bohm} vector potential
\begin{equation}\label{eq:AB}
	A(x,y)=\widetilde{\Psi} \left(\frac{-y}{x^2+y^2}, \frac{x}{x^2+y^2} \right), 
	\qquad \widetilde \Psi\in \R,
\end{equation}
and we denote by $\nabla_A:=\nabla-iA$ the magnetic gradient.
In particular, if $\widetilde\Psi\notin\Z$, then \eqref{eq:LW} gives a non-trivial 2D-Hardy inequality. Notice that the potential $A$ in \eqref{eq:AB} is very singular, since $A\notin L^2_{\text{loc}}(\R^2)$. Some examples in higher dimensions have been  recently introduced in \cite{FKLV}.

\medskip
\noindent
As for the Rellich inequality~\eqref{eq:Rellich-classical}, similar arguments lead to the statement of facts (i) and (ii) above. 
About (i), it is known that there are two cases of special interest. First, when $d=2$ inequality~\eqref{eq:Rellich-classical} still holds but only for functions $\psi\in C^\infty_0(\R^2\setminus\{0\})$ which satisfy the following orthogonality condition
\begin{equation}\label{eq:orth-cond}
	f_1(r):=\int_0^{2\pi} \psi(r,\theta)\overline{Y_1(\theta)}\, d\theta
	=0,
	\qquad Y_1(\theta):=e^{i\theta}.
\end{equation}
Second, when $d=4$, even if one works on the domain $H^2(\R^4\setminus\{0\})$, the inequality~\eqref{eq:Rellich-classical} gives a trivial contribution, as mentioned above. 
Indeed~\eqref{eq:Rellich-classical} descends from the following estimate (see~\cite[Section 7, pag. 94]{RB1969})
\begin{equation}\label{eq:Rellich-preliminary}
	\int_{\R^d} |\Delta \psi(x)|^2\, dx
	\geq 
	\frac{d^2(d-4)^2}{16}\int_{\R^d} \frac{|\psi(x)|^2}{|x|^4}\, dx
	+ p_0\int_{\R^d} \frac{|\psi(x)|^2}{|x|^4}\, dx,
	\qquad p_0:=\min_{k\in \N_0}\Big[c_k \Big(\frac{d(d-4)}{2} + c_k\Big)\Big],
\end{equation}
with $c_k:=k(k+d-2),$ $k\in \N_0$ being the eigenvalues of the Laplace-Beltrami operator $-\Delta_{\mathbb{S}^{d-1}}.$ If $d=4,$ then the first term in the right-hand side of~\eqref{eq:Rellich-preliminary} disappears, whereas $p_0=\min c_k^2=0.$ This gives the claimed trivial Rellich inequality in $d=4.$ If $d=2,$ then $p_0=\min k(k-2).$ Notice that $k^2(k^2-2)\geq 0$ if $k\neq 1,$ thus the Rellich inequality~\eqref{eq:Rellich-classical} holds also in $d=2$ with constant $C(2)=1$ as soon as $f_1(r)=0,$ \emph{i.e.} when $\psi$ satisfies~\eqref{eq:orth-cond}.    

\medskip
\noindent
Moving to the discussion about fact (ii), on the same line of the work by Laptev and Weidl~\cite{LW1999}, Evans and Lewis~\cite{EL2005} showed that, for $d=2,4$, the Rellich inequality
\begin{equation}\label{eq:Rellich-magnetic}
	\int_{\R^d} |\Delta_{A} \psi(x)|^2\, dx
	\geq \widetilde C(d) \int_{\R^d} \frac{|\psi(x)|^2}{|x|^4}\, dx
\end{equation} 
holds true for any $\psi\in H^2_A:=\left\{f\in L^2(\R^d):\int_{\R^d}|\Delta_A f|^2<\infty\right\}$.
Here $A$ is the Aharonov-Bohm potential in \eqref{eq:AB} when $d=2$, or a higher dimensional generalization if $d\geq3$ (see~\eqref{eq:AB-gen} below).  As for the constant $\widetilde C(d)$, we have $\widetilde C(2)=\min_{m\in \Z}((m+\widetilde\Psi)^2 -1)^2$ and $\widetilde C(4)=\min_{m\in \Z'}((m+\widetilde\Psi)^2 -1)^2,$ where $\Z'=\{m\colon (m+\widetilde \Psi)^2\geq 1\}.$
If  $\widetilde \Psi\in \Z,$ then $\widetilde C(2)=\widetilde C(4)=0.$ Moreover, when $d=2,$ if one assumes the orthogonality condition~\eqref{eq:orth-cond}, then $\widetilde C(2)=1$\footnote{By the gauge invariance, in the case $\widetilde{\Psi}\in \Z$, the Hamiltonian $-\Delta_A$ is unitarily equivalent to the free Hamiltonian. Thus $\widetilde C(2)=0$ if and only if $m=\pm1.$ Condition~\eqref{eq:orth-cond} ensures that the minimum is taken over $\Z\setminus\{-1,1\}.$ This yields $\widetilde C(2)=1.$}. 
%Thus Rellich's original results for $d=2,4$ are recovered.

\medskip
\noindent
As far as we know, improvements upon Hardy-Rellich inequalities~\eqref{eq:HardyRellich-classical} in the same style as above are still missing in the literature. The purpose of this paper is to fill this gap. Such improvements descend from a more general result, which is the main contribution of this paper. 
\begin{theorem}[Improved weighted Hardy-Rellich]\label{thm:main-general}
	In dimension $d\geq2$, let $\Lambda_\omega$ be a non-negative, self-adjoint operator with domain $\Dom(\Lambda_\omega)\subset L^2(\mathbb{S}^{d-1};d\omega)$. Assume that $\Lambda_\omega$ has purely discrete spectrum, consisting of isolated eigenvalues $\lambda_m,$ $m\in \mathcal{I}$ (repeated  according to multiplicity), which can accumulate only at infinity, with corresponding normalized eigenfunctions $u_m,$ $m\in \mathcal{I},$ being $\mathcal{I}$ a countable index set. Denote by
		$L_r:=-\frac{\partial^2}{\partial r^2} - \frac{d-1}{r}\frac{\partial}{\partial r}$,
	and define the non-negative operator 
	\begin{equation}\label{eq:operator}
		\mathcal{L}:=L_r + \tfrac{1}{r^2}\Lambda_\omega 
	\end{equation}		
	acting on the set
	\begin{equation}\label{eq:domain}
		\Dom(\mathcal{L}):=
		\{\psi\colon \psi \in C^{\infty}_0(\R^d\setminus\{0\}),
		\, \psi(r,\cdot)\in \Dom(\Lambda_\omega)\, \text{for }  r>0,
		\, \text{and } \mathcal{L}\psi\in L^2(\R^d)
		\}.
	\end{equation}
Let $\alpha\in \R.$ Then, for all $\psi\in \Dom(\mathcal{L})$ such that $|\cdot|^{-\alpha/2}\mathcal{L}\psi \in L^2(\R^d)$ we have
\begin{equation}\label{eq:main}
	\int_{\R^d} \frac{|\mathcal{L}\psi(x)|^2}{|x|^\alpha}\, dx
	\geq C(d,\alpha) \int_{\R^d} \frac{\mathcal{D}\psi(x)}{|x|^{\alpha+2}}\, dx,
\end{equation} 
where $\mathcal{D}$ is the first-order operator defined by
%\begin{equation*}
	$\mathcal{D}\psi:=\big|\frac{\partial \psi}{\partial r}\big|^2 + \frac{1}{r^2}|\Lambda_\omega^{1/2}\psi|^2
	\footnote{ $\Lambda_\omega^{1/2}$ is the square root of the non-negative, self-adjoint operator $\Lambda_\omega.$ This operator exists and is unique by the functional calculus (see, for example,~\cite[Prop.5.13]{Schmu}). In particular, $\Lambda_\omega^{1/2}u_m=\sqrt{\lambda_m}u_m,$ where $u_m,$ $m\in \mathcal{I}$ are the eigenfunctions of $\Lambda_\omega$ and $\lambda_m$ the corresponding eigenvalues.}$,
%\end{equation*}
and where $C(d, \alpha)$ is given by
\begin{equation}\label{eq:C(d,alpha)}
	C(d,\alpha)=
	\begin{cases}\displaystyle
	\min_{m\in I}\frac{(4\lambda_m + (d+\alpha)(d-\alpha-4))^2}{4(4\lambda_m +(d-\alpha-4)^2)}, \qquad &\text{if }d-\alpha-4\neq0,\vspace{0.2cm}\\
	\min \Big((d-2)^2; \min_{\substack{m\in \mathcal{I}\\ \lambda_m\neq 0}}\lambda_m \Big),
	&\text{if } d-\alpha-4=0.
	\end{cases}
\end{equation}
\end{theorem}
\begin{remark}\label{rem:cut}
Due to the general statement of Theorem~\ref{thm:main-general}, which aims at covering \emph{any} dimension $d\geq 2$ and \emph{any} power-weight $\alpha\in \R,$ we needed to restrict ourselves to considering functions $C^\infty_0(\R^d\setminus \{0\})$. Despite that, it is clear that, by density arguments, if one restricts to particular situations according to the values of $d$ and $\alpha,$ then the assumption of cutting the origin can be dropped. For example, if $d\geq 3,$ $\alpha=0$ and $\mathcal{L}=-\Delta,$ then~\eqref{eq:main} holds for any $\psi\in C^\infty_0(\R^d)$ (see~\cite{Cazacu2019}). 
\end{remark}

\begin{remark}
We stress that the right hand side of~\eqref{eq:main} can be written in terms of the \emph{Carré du Champ} associated to $\mathcal{L}.$
\begin{definition}\label{def:carreduchamp}
	Given $\mathcal L$ a linear operator on $L^2(\R^d;\C),$ the Carré du Champ associated to $\mathcal{L}$ is the sesquilinear form $\Gamma$ on $C^\infty_0(\R^d)\times C^\infty_0(\R^d)$ defined by
	\begin{equation*}
		2\Gamma(\psi, \phi)=\overline{\psi}\mathcal{L}\phi + \overline{\mathcal{L}\psi}  \phi - \mathcal{L}(\overline{\psi} \phi).
	\end{equation*}
	In particular
	\begin{equation*}
		2\Gamma(\psi):=2\Gamma(\psi,\psi)
		= 2\Re(\overline{\psi}\mathcal{L}\psi) - \mathcal{L}|\psi|^2.
	\end{equation*}
\end{definition} 
\noindent
Using integration by parts one sees that
\begin{equation}\label{eq:cdc-corr}
	\int_{\R^d} \Gamma(\psi)|x|^\beta\, dx
	=\int_{\R^d} |\partial_r \psi|^2 |x|^{\beta}\, dx
	+ \int_{\R^d} |\Lambda_\omega^{1/2}\psi|^2 |x|^{\beta-2}\, dx 
	- \frac{1}{2}\int_{\R^d} \frac{|\psi|^2}{|x|^2}\Lambda_\omega |x|^\beta\, dx,
\end{equation}
in other words the right hand side of~\eqref{eq:main} can be written in terms of the Carré du Champ provided that an angular correction is added. More specifically, we have 
\begin{equation*}
	\int_{\R^d} \mathcal{D}(\psi) |x|^\beta\, dx
	=\int_{\R^d} \Gamma(\psi)|x|^\beta\, dx + \frac{1}{2}\int_{\R^d} \frac{|\psi|^2}{|x|^2}\Lambda_\omega |x|^\beta\, dx.
\end{equation*}
For the sake of completeness we show~\eqref{eq:cdc-corr}. From Definition~\ref{def:carreduchamp} one has
\begin{equation}\label{eq:Gamma-pre}
\begin{split}
\int_{\R^d} \Gamma(\psi)|x|^\beta\, dx
&= \int_{\R^d} \Re\big (\overline{\psi}\mathcal{L}\psi\big ) |x|^{\beta}\, dx
-\frac{1}{2} \int_{\R^d} \mathcal{L}(|\psi|^2) |x|^{\beta}\, dx\\
&=\int_{\R^d} \Re\big (\overline{\psi}L_r \psi\big ) |x|^{\beta}\, dx
+ \int_{\R^d} \Re\big (\overline{\psi}\Lambda_\omega\psi\big ) |x|^{\beta-2}\, dx
-\frac{1}{2} \int_{\R^d} |\psi|^2 \mathcal{L} |x|^{\beta}\, dx,
\end{split}
\end{equation}
where in the last identity we have used that $\mathcal{L}$ can be written as $\mathcal{L}=L_r + \frac{1}{r^2}\Lambda_\omega$ and that $\mathcal{L}$ is self-adjoint.
Using integration by parts, one easily checks that
\begin{equation}\label{eq:int-parts}
	\langle f, L_r g\rangle_{L^2(\R^d)}=\langle \partial_r f, \partial_r g\rangle_{L^2(\R^d)},
	\qquad \text{and} \qquad
	\langle u, \Lambda_\omega v \rangle_{L^2(\mathbb{S}^{d-1})}
	=\langle \Lambda_\omega^{1/2} u, \Lambda_\omega^{1/2} v \rangle_{L^2(\mathbb{S}^{d-1})}.		
\end{equation}
Using~\eqref{eq:int-parts} and $L_r |x|^\beta=-\beta(d+\beta -2)|x|^{\beta -2}$ in~\eqref{eq:Gamma-pre} we get
\begin{multline*}
\int_{\R^d} \Gamma(\psi)|x|^\beta\, dx
=\int_{\R^d} |\partial_r \psi|^2 |x|^{\beta}\, dx 
+ \beta \int_{\R^d} \Re(\overline{\psi}\partial_r \psi)|x|^{\beta-1}\, dx\\
+ \int_{\R^d} |\Lambda_\omega^{1/2}\psi|^2 |x|^{\beta-2}\, dx
- \frac{\beta(d+\beta-2)}{2}\int_{\R^d} |\psi|^2 |x|^{\beta-2}
-\frac{1}{2} \int_{\R^d} \frac{|\psi|^2}{|x|^2} \Lambda_\omega |x|^{\beta}\, dx.
\end{multline*}
Integrating by parts with respect to the radial variable $r,$ the second term cancels with the last but one term and thus~\eqref{eq:cdc-corr} follows.
\end{remark}

\medskip
\noindent
Theorem \ref{thm:main-general} is stated in dimension $d\geq2$, in order to describe a non-trivial contribution given by the angular operator $\Lambda_\omega$. Anyway an analogous result holds true in $d=1$ as well. More precisely, the following theorem is an immediate consequence of the classical 1D-weighted Hardy inequality (see~\eqref{eq:1d-Hardy} below) applied to $\psi'.$ 
\begin{theorem}[1D-weighted Hardy-Rellich]\label{thm:1D}
Let $d=1.$ Then for all $\psi\in C^\infty_0(\R\setminus \{0\}),$ we have
\begin{equation}\label{eq:1dHardy-Rellich}
	\int_\R \frac{|\psi''(x)|^2}{|x|^\alpha}\, dx\geq \frac{(\alpha+1)^2}{4}\int_\R \frac{|\psi'(x)|^2}{|x|^{\alpha+2}}\, dx.
\end{equation}
\end{theorem}
\begin{remark}
	Notice that in the weighted free case $\alpha=0,$~\eqref{eq:1dHardy-Rellich}  gives the claimed Hardy-Rellich inequality~\eqref{eq:HardyRellich-classical} for $d=1$ with $C(1)=1/4.$
\end{remark}

\medskip
\noindent
We also claim that the constant $C(d,\alpha)$ in~\eqref{eq:C(d,alpha)} is sharp and not attained. To show this in the complete generality, we construct a minimizing sequence which is suitably supported far away from the origin. To this aim, given $\epsilon>0$, we introduce a smooth cut-off function $g_\epsilon\in C^\infty_0(\R^+)$ such that 
\begin{equation}\label{eq:cutoff}
	g_\epsilon (r)=
	\begin{cases}
		0, \quad &\text{if } 0\leq r\leq \epsilon\; \text{ or } r\geq1/\epsilon,\\
		1, &\text{if } 2\epsilon\leq r\leq 1/2\epsilon,
		%&0, &\text{if } r\geq 2,
	\end{cases}
\end{equation}
$0\leq g_\epsilon\leq 1$ in $0\leq r<\infty$, and
\begin{equation*}
\begin{cases}
	|g_\epsilon'(r)|\leq \frac{c}{\epsilon}, 
	\quad
	\hspace{0.08cm} |g_\epsilon''(r)|\leq \frac{c}{\epsilon^2},
\quad &\text{for } \epsilon\leq r\leq 2\epsilon,\\
	|g_\epsilon'(r)|\leq c\epsilon, 
	\quad
	|g_\epsilon''(r)|\leq c\epsilon^2,
	\quad &\text{for } 1/2\epsilon\leq r\leq 1/\epsilon,\\
\end{cases}
\end{equation*}   
for some constant $c>0$. 
We have the following result.
\begin{theorem}[Optimality of \eqref{eq:main}]\label{thm:minimizing}
In dimension $d\geq 2$, for any $\epsilon$, define
\begin{equation}\label{eq:minimizing}
	\psi_\epsilon(x):=
	\begin{cases}
	|x|^{\frac{-(d-4) +\alpha}{2}} g_\epsilon(|x|) u_{m_0}\big(\tfrac{x}{|x|}\big), \qquad  &\text{if } d-\alpha-4\neq 0 \text{  \;or\; } C(d,\alpha)=\lambda_{m_0},\\
	h_\epsilon(|x|), & \text{if } d-\alpha-4=0 \text{ and } C(d,\alpha)=(d-2)^2,
	\end{cases}
\end{equation}
where $m_0\in \mathcal{I}$ is a minimizing index in~\eqref{eq:C(d,alpha)}, $u_{m_0}$ is the eigenfunction corresponding to the eigenvalue $\lambda_{m_0}$ of $\Lambda_{\omega}$ and $h_\epsilon$ is defined such that
\begin{equation}\label{eq:h_eps}
	h_\epsilon'(r)=r^{-1}g_\epsilon(r), \qquad r=|x|.
\end{equation}
 Then $\{\psi_\epsilon\}_{\epsilon>0}\subset \Dom(\mathcal{L})$ is a minimizing sequence for $C(d,\alpha),$ \emph{i.e.}
\begin{equation*}
\frac{\int_{\R^d} |\mathcal{L}\psi_\epsilon(x)|^2/|x|^\alpha\, dx}{\int_{\R^d} \mathcal{D}(\psi_\epsilon)(x)/|x|^{\alpha+2}\, dx}
\searrow
C(d,\alpha),
\; \text{as } \epsilon \searrow 0.
\end{equation*}
Besides, the constant $C(d,\alpha)$ is not attained in $\Dom(\mathcal{L}).$
\end{theorem}

\medskip
\noindent
We present now some interesting particular cases of Theorem~\ref{thm:main-general}, which show the analogous improvements as the above mentioned ones for the Hardy and the Rellich inequalities.
\begin{theorem}\label{thm:general-electric}
	Assume $d\geq 2.$ Let $a\in L^\infty(\mathbb{S}^{d-1}; d\theta)$ be a non-negative real-valued function and consider the non-negative operator $-\Delta_{a(\theta)}:=-\Delta + \frac{a(\theta)}{|x|^2}.$ Then for all $\psi\in C^\infty_0(\R^d\setminus \{0\}),$
	\begin{equation}\label{eq:general-electric}
		\int_{\R^d} \frac{|-\Delta_{a(\theta)}\psi(x)|^2}{|x|^\alpha}\, dx \geq C_a(d,\alpha)\Bigg[ \int_{\R^d} \frac{|\nabla \psi(x)|^2}{|x|^{\alpha + 2}}\, dx + \int_{\R^d} a(\theta) \frac{|\psi(x)|^2}{|x|^{\alpha+ 4}}\, dx\Bigg],
	\end{equation} 
	where $C_{a}(d,\alpha)$ is given by 
	\begin{equation*}
	C_{a}(d,\alpha)=
	\begin{cases}\displaystyle
		\min_{k\in \N_0}\frac{(4\mu_k + (d+\alpha)(d-\alpha-4))^2}{4(4\mu_k +(d-\alpha-4)^2)}, \qquad &\text{if }d-\alpha-4\neq0,\vspace{0.2cm}\\
	\min \big((d-2)^2; \min_{\substack{k\in \N_0\\ \mu_k\neq 0}}\mu_k \big),
	 &\text{if } d-\alpha-4=0.		
	 \end{cases}
\end{equation*}
Here $\mu_k,$ with $k=0,1,\dots$ are the discrete eigenvalues of the angular operator $-\Delta_{\mathbb{S}^{d-1}} + a(\theta).$ Moreover $\mu_0\geq \ess \inf_{\mathbb{S}^{d-1}} a(\theta).$
\end{theorem}

\begin{remark}
	Notice that the right hand side of~\eqref{eq:general-electric} is exactly the (weighted) quadratic form associated to $-\Delta_{a(\theta)}.$ The same holds for the particular case of Corollary~\ref{cor:a-d2} below.
\end{remark}
\noindent
An interesting corollary of the above result is the following, with $d=2$, and $a(\theta)\equiv a>0$.
 \begin{corollary}\label{cor:a-d2}
 	Assume $d=2.$ Let $a\geq 0$ and consider the non-negative operator $-\Delta_a:=-\Delta + \tfrac{a}{|x|^2}.$ Then for all $\psi\in C^\infty_0(\R^2\setminus \{0\}),$
	\begin{equation}\label{eq:prot-electr}
		\int_{\R^2}|-\Delta_a \psi(x)|^2\, dx
		\geq C_a \Bigg[ \int_{\R^2} \frac{|\nabla \psi(x)|^2}{|x|^{\alpha + 2}}\, dx + a \int_{\R^2} \frac{|\psi(x)|^2}{|x|^{\alpha+ 4}}\, dx\Bigg].
	\end{equation} 
	The constant $C_a$ in~\eqref{eq:prot-electr} is given by
	\begin{equation}\label{eq:C_a}
		C_a=\min_{k\in \N_0} \frac{(k^2+a-1)^2}{k^2+a+1}.
	\end{equation}
\end{corollary}
\noindent
Notice that if $a>1,$ then $C_a=\tfrac{(a-1)^2}{a+1}>0.$

\medskip
\noindent
Another consequence of Theorem~\ref{thm:main-general} goes in the direction of the results of Laptev-Weidl~\cite{LW1999} and Evans-Lewis~\cite{EL2005}. In order to state the next result we generalize the definition  of Aharonov-Bohm type potentials to any dimension $d\geq 2:$ for $(x_1, x_2, \dots, x_d)\in \R^d \setminus\{x_d=x_{d-1}=0\}$ it is defined to be the vector field
\begin{equation}\label{eq:AB-gen}
A(x_1, x_2, \dots, x_d)
=\widetilde \Psi \Bigg(\underbrace{0,0, \dots, 0}_{d-2}, -\frac{x_d}{x_{d-1}^2 + x_{d}^2}, \frac{x_{d-1}}{x_{d-1}^2 + x_{d}^2}\Bigg),
\qquad \widetilde{\Psi}\in \R.
\end{equation}

\begin{theorem}\label{thm:HB-anydimension}
	Let $d\geq 2$ and let $A$ be the Aharonov-Bohm type vector potential given by~\eqref{eq:AB-gen}.
%	 for $d\geq 3,$ respectively~\eqref{eq:AB} for $d=2,$ and let $\widetilde \Psi$ be the corresponding magnetic flux given by~\eqref{eq:magnetic-flux}. 
Then for all $\psi\in C^\infty_0(\R^d\setminus \{0\}),$
	\begin{equation}\label{eq:HR-anydimension}
		\int_{\R^d}\frac{|\Delta_A\psi(x)|^2}{|x|^\alpha}\, dx\geq C_{\textup{AB}}(d,\alpha)\int_{\R^d} \frac{|\nabla_{\!A}\psi(x)|^2}{|x|^{\alpha+2}} \, dx.  	
	\end{equation}
	The constant $C_\textup{AB}(d,\alpha)$ is given by
	\begin{equation*}
	C_\textup{AB}(d,\alpha)=
	\begin{cases}\displaystyle
	%\label{eq:C_AB}
		\min_{m\in \Z'} \frac{\left( 4(m+ \widetilde \Psi)(m + \widetilde \Psi + d-2) + (d-4-\alpha)(d+\alpha)\right)^2}{4(4(m+ \widetilde \Psi)(m + \widetilde \Psi + d-2) + (d-4-\alpha)^2) },
		 \qquad &\text{if }d-\alpha-4\neq 0,\vspace{0.2cm}\\
		\min \Big((d-2)^2;\min \{ (m+ \widetilde \Psi)^2(m + \widetilde \Psi + d-2)^2 \mid m\in\Z', m+\widetilde \Psi\neq 0, 2-d \}\Big),   &\text{if }d-\alpha-4=0,
		\end{cases}
	\end{equation*}
	where $\Z':=\{m\in \Z\colon m\leq 2-d-\widetilde \Psi\, \text{or } m\geq -\widetilde \Psi\}.$
\end{theorem}

\begin{remark}\label{rmk:mgrad-cdc}
	Notice that in the specific situation of Theorem~\ref{thm:HB-anydimension} (and Corollary~\ref{thm:Hardy-Rellich_magnetic} below), in the Hardy-Rellich inequality~\eqref{eq:HR-anydimension} it appears the magnetic gradient instead of the first order operator $\mathcal{D}$ as in the general case. Indeed, one can check that the integral identity $\int_{\R^d}\mathcal{D}(\psi)/|x|^{\alpha+2}=\int_{\R^d}|\nabla_A\psi|^2/|x|^{\alpha +2}$ (refer to the proof of Theorem~\ref{thm:HB-anydimension} for more clarifications).
\end{remark}

\medskip
\noindent
In the case $d=2$ and $\alpha=0,$ Theorem~\ref{thm:HB-anydimension} reduces to the following result, in the same style as~\eqref{eq:LW} and~\eqref{eq:Rellich-magnetic}.

\begin{corollary}\label{thm:Hardy-Rellich_magnetic}
	Assume $d=2.$ Let $A$ be the Aharonov-Bohm (AB) type vector potential given by~\eqref{eq:AB}. 
%	and let $\widetilde \Psi$ be the corresponding magnetic flux given by~\eqref{eq:magnetic-flux}. 
	Then for all $\psi\in C^\infty_0(\R^2\setminus \{0\}),$
	\begin{equation}\label{eq:Hardy-Rellich_magnetic}
		\int_{\R^2} |\Delta_A \psi(x)|^2\, dx\geq C_{\textup{AB}}\int_{\R^2} \frac{|\nabla_{\!A}\psi(x)|^2}{|x|^2}\, dx.
	\end{equation}
	The constant $C_{\textup{AB}}$ in~\eqref{eq:Hardy-Rellich_magnetic} is given by
	\begin{equation*}
		C_{\textup{AB}}=\min_{m\in \Z} \frac{((m+ \widetilde \Psi)^2-1)^2}{(m+ \widetilde \Psi)^2+1}.
	\end{equation*}
\end{corollary}
\noindent
Notice that $C_\textup{AB}=0$ if and only if $\widetilde{\Psi}\in \Z$, which fits with the fact that no Hardy-Rellich inequality holds in dimension $d=2$ for the free Hamiltonian.

\medskip
\noindent
If we assume $\widetilde \Psi\in \Z$ then Theorem~\ref{thm:HB-anydimension} covers the weighted Hardy-Rellich inequalities already available for the free Hamiltonian (see~\cite{TZ2007,Beckner2008,GM2011,Cazacu2019,HT2021,HT2021arXiv,Ham}). More precisely, we have the following corollary.
\begin{corollary}\label{cor:free}
	Let $d\geq 2.$ Then for all $\psi\in C^\infty_0(\R^d\setminus \{0\})$
	\begin{equation}\label{eq:weighted-Hardy-Rellich}
		\int_{\R^d}\frac{|\Delta\psi(x)|^2}{|x|^\alpha}\, dx\geq C(d,\alpha)\int_{\R^d} \frac{|\nabla \psi(x)|^2}{|x|^{\alpha+2}} \, dx.  
	\end{equation}
	The constant $C(d,\alpha)$ is given by
	\begin{equation}\label{eq:constant-classical}
	C(d,\alpha)=
	\begin{cases}\displaystyle
		\min_{k\in \N_0} \frac{\left(4k(k + d-2) + (d-4-\alpha)(d+\alpha)\right)^2}{4(4k(k + d-2)+(d-4-\alpha)^2)}, 
		\qquad &\text{if }d-\alpha-4\neq0,\\
		\min ( (d-2)^2; (d-1)),
		&\text{if } d-\alpha-4=0.
	\end{cases}
	\end{equation}
\end{corollary} 

\begin{remark}
	In the specific case of Corollary~\ref{cor:free} one easily checks that $\int_{\R^d}\mathcal{D}(\psi)/|x|^{\alpha+2}=\int_{\R^d}|\nabla \psi|^2/|x|^{\alpha+2}.$ This can be seen from~\eqref{eq:cdc-corr}, indeed one checks easily that the Carré du Champ associated to the classical Laplacian is $\Gamma(\psi)=|\nabla \psi|^2$ and moreover, since in this case $\Lambda_\omega=\Delta_{\mathbb{S}^{d-1}},$ where $\Delta_{\mathbb{S}^{d-1}}$ denotes the Laplace-Beltrami operator, the last term in~\eqref{eq:cdc-corr} cancels.
\end{remark}
\begin{remark}\label{rmk:grad}
The value of the constant $C(d,\alpha)$ in~\eqref{eq:weighted-Hardy-Rellich} has been largely investigated in the aforementioned works~\cite{TZ2007,Beckner2008,GM2011,Cazacu2019,HT2021}. There, according to the relation between the relevant parameters, namely the dimension $d,$ the order of the weight-power $\alpha$ and the non-negative integer $k,$ a more explicit description has been provided in different cases.
Here we will describe the behavior of the constant only in the weight-free case, namely the original case $\alpha=0,$ and we show that $C(d, 0)$ in~\eqref{eq:weighted-Hardy-Rellich} coincides with the best constant $C(d)$ in~\eqref{eq:HardyRellich-classical}. Nevertheless we stress that in the case $\alpha\neq 0$ we recover the previous available results in~\cite{GM2011,TZ2007,HT2021}.
First of all, if $\alpha=0$ then one immediately has from the second expression in~\eqref{eq:constant-classical} that $C(4,0)=3.$
When $\alpha=0$ and $d\neq 4$ one needs to study the first in~\eqref{eq:constant-classical}. Plugging $\alpha=0,$ the first expression in~\eqref{eq:constant-classical}  becomes
\begin{equation}\label{eq:constant-classical-zero}
	C(d,0)=\min_{k\in \N_0} \frac{\left(4  c_k + d(d-4)\right)^2}{4(4c_k + (d-4)^2)}, 
	\qquad c_k=k(k+d-2).
\end{equation}
Studying the minimum $x_0$ of the function $\frac{(4x+d(d-4))^2}{4(4x+ (d-4)^2)}$ for $x\geq 0,$ one sees that $c_0\leq x_0\leq c_1,$ equivalently $0\leq x_0\leq d-1$ (see definition of $c_k$ in~\eqref{eq:constant-classical-zero}). Thus the value of $C(d,0)$ in~\eqref{eq:constant-classical-zero} depends only on $k=0$ and $k=1.$ One can easily check that for $d\geq 5$ the minimum is obtained for $k=0,$ yielding $C(d,0)=\frac{d^2}{4}.$ Instead in lower dimensions, namely $d\in\{2,3\},$ then this is achieved for $k=1.$ This gives $C(2,0)=0,$ $C(3,0)=\tfrac{25}{36}.$ Thus $C(d,0)$ in~\eqref{eq:weighted-Hardy-Rellich} equals $C(d)$  in~\eqref{eq:HardyRellich-classical} as claimed. We stress that in the two-dimensional setting $d=2$ no non-trivial inequalities are available, indeed $C(2,0)=0,$ unless one restricts the domain of validity of inequality~\eqref{eq:weighted-Hardy-Rellich} to functions $\psi\in C^\infty_0(\R^2\setminus \{0\})$ which satisfy~\eqref{eq:orth-cond}. In this case then the minimum in~\eqref{eq:constant-classical-zero} is taken over $\N_0\setminus \{1\}$ and thus, due to the reasoning above, it is achieved for $k=0$ giving $C(2,0)=1.$ This means that in $d=2$ a non-trivial Hardy-Rellich inequality holds true if one restricts to a smaller set of function. We stress that as far as we know this simple two dimensional property was not observed before elsewhere.
\end{remark}
		
\medskip
\noindent
A further example is given by magnetic monopoles in $\R^3.$ This model has been intensively studied in the last decades (see~\cite{CT2010}). More recently, Frank and Loss~\cite{FL20} considered it as an example of (non-standard) magnetic field that supports zero mode for the three dimensional Dirac equation.  
For a magnetic monopole at the origin, the vector field $A$ takes the form  
\begin{equation}\label{eq:A-monopole}
	A(x,y,z)=g \frac{(-y,x,0)}{r(r+z)}, \qquad r=\sqrt{x^2+y^2+z^2},
	\qquad (x,y,z)\in \R^3\setminus\{(0,0,z)\mid z\leq0\}
\end{equation}
with a parameter $g$ representing the monopole strength. The corresponding magnetic field is given by
\begin{equation*}
	B(x,y,z)=\curl A= g \frac{(x,y,z)}{r^3}.
\end{equation*}
As a consequence of Theorem~\ref{thm:main-general}, the following result holds.
\begin{theorem}\label{thm:monopole}
	Let $d=3$ and assume $g\geq 1/2$ in~\eqref{eq:A-monopole}. Then for all $\psi \in C^\infty_0(\R^3\setminus \{0\}),$
	\begin{equation*}
		\int_{\R^3} \frac{|\Delta_A \psi(x)|^2}{|x|^\alpha}\, dx 
		\geq C_\textup{mon}(\alpha)\int_{\R^3} \frac{\mathcal{D} (\psi)(x)}{|x|^{\alpha+2}}\, dx. 
	\end{equation*}
	The constant $C_\textup{mon}(\alpha)$ is given by
	\begin{equation*}
		C_\textup{mon}(\alpha)=\min_{\substack{k=2(|g|+l),\\ 
		l\in \N_0}} 
		\frac{(k(k+2)-4g^2-(\alpha+1)(\alpha+ 3))^2}{4(k(k+2)-4g^2 +(\alpha+1))}.
	\end{equation*}
\end{theorem}
\begin{remark}
We could not find in the literature Hardy and Rellich inequalities involving magnetic monopoles. Nevertheless, the same approach (even simplified) we use to prove Hardy-Rellich inequalities for this model, namely Theorem~\ref{thm:monopole}, can be adopted to establish improvements of these more classical inequalities. 
\end{remark}

\noindent
Using a more direct strategy than the one used to prove Theorem~\ref{thm:main-general}, the following weighted Hardy-type inequalities for the first order operator $\mathcal{D}$ associated to $\mathcal{L}=L_r + \frac{1}{r^2}\Lambda_\omega$ are easily obtained.
\begin{theorem}\label{thm:Hardy-cdc}
Assume that the hypotheses of Theorem~\ref{thm:main-general} are satisfied. Let $\beta \in \R.$ Then for all $\psi \in \Dom(\mathcal{L})$ such that $|\cdot|^{-\beta/2}\mathcal{D}(\psi)^{1/2}\in L^2(\R^d)$ we have
\begin{equation}\label{eq:Hardy-cdc}
\int_{\R^d} \frac{\mathcal{D}(\psi)(x)}{|x|^{\beta}}\, dx
\geq C_{\mathcal{D}}(d,\beta)
\int_{\R^d} \frac{|\psi(x)|^2}{|x|^{\beta+2}}\, dx,
\end{equation}
where $C_\mathcal{D}(d,\beta)$ is given by
\begin{equation}\label{eq:CGamma}
C_\mathcal{D}(d, \beta)=\min_{m\in \mathcal{I}}\Bigg\{\lambda_m + \Big(\frac{d-\beta-2}{2}\Big)^2 \Bigg\}.
\end{equation}
\end{theorem}
\begin{remark}
Notice that~\eqref{eq:Hardy-cdc} has as particular cases the classical weighted Hardy inequalities with optimal constants (just take $\mathcal{L}=-\Delta$ and notice that in this case $\mathcal{D}(\psi)=|\nabla \psi|^2$) and the optimal magnetic Hardy inequalities for Aharonov-Bohm magnetic fields (take $\mathcal{L}:=-\Delta_A,$ with $A$ as in~\eqref{eq:AB-gen} and use that $\int_{\R^d} \mathcal{D}(\psi)/|x|^\beta=\int_{\R^d} |\nabla_A \psi|^2/|x|^\beta$ (see also Remark~\ref{rmk:mgrad-cdc}))
\end{remark}
\begin{remark}
Combining the Hardy inequality~\eqref{eq:Hardy-cdc} in Theorem~\ref{thm:Hardy-cdc} and the Hardy-Rellich inequality~\eqref{eq:main} in Theorem~\ref{thm:main-general} one gets easily the following weighted Rellich inequalities in the spirit of Evans and Lewis~\cite{EL2005}:
\begin{equation}\label{eq:Rellich-cdc}
	\int_{\R^d} \frac{|\mathcal{L}\psi(x)|^2}{|x|^\alpha}\, dx\geq \widetilde{C}(d,\alpha) \int_{\R^d} \frac{|\psi(x)|^2}{|x|^{\alpha + 4}}\, dx,
\end{equation}
where $\widetilde{C}(d,\alpha)=C(d,\alpha) C_\mathcal{D}(d, \alpha+2),$ where $C(d,\alpha)$ is as in~\eqref{eq:C(d,alpha)} and $C_\mathcal{D}(d, \alpha+2)$ is as in~\eqref{eq:CGamma}. In general, it is not easy to see whether $\widetilde{C}(d,\alpha)$ equals the optimal constant in~\cite{EL2005}, nevertheless in the specific case of the Laplacian, namely for $\mathcal{L}=-\Delta,$ and for $\alpha=0,$ one checks that $\widetilde{C}(d,0)=d^2(d-4)^2/16$ (see Remark~\ref{rmk:grad}). In other words, inequality~\eqref{eq:Rellich-cdc} coincides with the classical Rellich inequality with optimal constant.  
\end{remark}
\medskip
\noindent
The paper is organized as follows: we give the proof of the main result Theorem~\ref{thm:main-general} and of Theorem~\ref{thm:Hardy-cdc}
%and its one dimensional counterpart Theorem~\ref{thm:1D} 
in the next Section~\ref{sec:main}. 
The optimality, as stated in Theorem~\ref{thm:minimizing}, is shown in Section~\ref{sec:optimality}. In Section~\ref{sec:consequences} we show how to get Theorem~\ref{thm:general-electric}, Theorem~\ref{thm:HB-anydimension} and Theorem~\ref{thm:monopole} from the general result Theorem~\ref{thm:main-general}. 
%Some additional material for the reader is contained in the Appendix~\ref{ap:AB}.

%%---------------------------%
\subsection*{Acknowledgments}
%%---------------------------%
The idea of this project came out during the CIRM conference on \emph{``Mathematical aspects of the physics with non-self-adjoint operators: 10 years after"} held in Marseille in February 2021. B.C. and L.C. would like to express their gratitude to the organizers of the conference L. Boulton, D. Krej\v ci\v r\'ik and P. Siegl for the chance of being part of this stimulating event.
B. C. is member of GNAMPA (INDAM) and he is supported by Fondo Sociale Europeo – Programma Operativo Nazionale Ricerca e Innovazione 2014-2020, progetto PON: progetto AIM1892920-attivit\`a 2, linea 2.1.
The research of L.C. is supported by the Deutsche Forschungsgemeinschaft (DFG) through CRC 1173. The authors are very grateful to the anonymous referee for the comments/suggestions on the preliminary version of this paper which highly improve the quality of the manuscript.

\section{Proof of Theorem~\ref{thm:main-general} 
and Theorem~\ref{thm:Hardy-cdc}
}\label{sec:main}
We start with the proof of Theorem~\ref{thm:main-general}. Inspired by the arguments in \cite{Cazacu2019}, we introduce
a suitable orthonormal basis decomposition of the functions in the domain $\Dom(\mathcal{L})$ (see~\eqref{eq:domain}) of the operator $\mathcal{L},$ which is reminiscent of the classical spherical harmonics decomposition used on the case of the Laplacian: since the spectrum of $\Lambda_\omega$ is assumed to be discrete, its normalized eigenvectors $u_m,$ $m\in \mathcal{I}$ (with eigenvalues $\{\lambda_m\}_{m\in \mathcal{I}}$ repeated according to multiplicity) form an orthonormal basis of $L^2(\mathbb{S}^{d-1};d\omega).$ Thus one can expand any $\psi\in \Dom(\mathcal{L})$ as
	\begin{equation}\label{eq:decomposition}
		\psi(x)=\psi(r,\omega)=\sum_{m\in \mathcal{I}}  f_m(r) u_m(\omega), 
	\end{equation} 
	where the coefficients $f_m\in C^\infty_0(\R^+)$ are computed by projecting $\psi$ onto each basis eigenfunction $u_m,$ $m\in \mathcal{I},$ \emph{i.e.}
	\begin{equation}\label{eq:coefficients}
		f_m(r):=\int_{\mathbb{S}^{d-1}}\psi(r,\omega)\overline{u_m(\omega)}\, d\omega.
	\end{equation}
The decomposition in~\eqref{eq:decomposition} reduces matters to a 1D-problem: indeed, the following lemma holds.

\begin{lemma}\label{lemma:1d-reduction}
Let $\psi\in \Dom(\mathcal{L}).$ Then the following identities hold true
\begin{align}
	\label{eq:second-order}
	&
	\begin{multlined}
	\int_{\R^d} \frac{|\mathcal{L}\psi(x)|^2}{|x|^\alpha}\, dx=
	\sum_{m\in \mathcal{I}}
	\left \{\int_0^\infty |f_m''(r)|^2r^{d-\alpha-1}\, dr 
	+ [(d-1)(\alpha +1) + 2\lambda_m]\int_0^\infty |f_m'(r)|^2 r^{d-\alpha-3} \,dr \right.\\
	\left.+ \lambda_m [(\alpha + 2)(d-\alpha -4) + \lambda_m] \int_0^\infty |f_m(r)|^2 r^{d-\alpha-5} \,dr
	\right \},
	\end{multlined}
	\\
	\label{eq:first-order}
	&\int_{\R^d} \frac{\mathcal{D}(\psi)(x)}{|x|^{\alpha+2}}\, dx=
	\sum_{m\in \mathcal{I}}
	\left \{\int_0^\infty |f_m'(r)|^2r^{d-\alpha-3}\, dr 
	+ \lambda_m \int_0^\infty |f_m(r)|^2 r^{d-\alpha-5} \,dr
	\right \}.	
\end{align} 
Here $\mathcal{D}$ is the first order operator defined as $\mathcal{D}(\psi)=|\partial_r \psi|^2 + \frac{1}{r^2}|\Lambda_\omega^{1/2}\psi|^2$  and $f_m(r),$ $m\in \mathcal{I}$ are the coefficients introduced in~\eqref{eq:coefficients}.
\end{lemma}
\begin{proof}[Proof of Lemma~\ref{lemma:1d-reduction}]
	Even though identity~\eqref{eq:second-order} can be already found in\cite{EL2005}, for the reader's convenience in the following we prove both~\eqref{eq:second-order} and~\eqref{eq:first-order}.
	Since $\mathcal{L}=L_r + \frac{1}{r^2}\Lambda_\omega$ one easily has 
	\begin{equation}\label{eq:pre-sec}
		\int_{\R^d} \frac{|\mathcal{L}\psi(x)|^2}{|x|^\alpha}\, dx
		=\int_{\R^d} \frac{|L_r\psi(x)|^2}{|x|^\alpha}\, dx + \int_{\R^d} \frac{|\Lambda_\omega\psi(x)|^2}{|x|^{\alpha+4}}\, dx + 2 \Re \int_{\R^d} \frac{L_r \psi(x) \overline{\Lambda_\omega \psi(x)}}{|x|^{\alpha + 2}}.
	\end{equation}	
	Let us consider the right hand side of~\eqref{eq:pre-sec}. From the decomposition~\eqref{eq:decomposition} one has
	\begin{equation*}
		L_r\psi(x)= \sum_{m\in \mathcal{I}} L_r f_m(r) u_m(\omega).
	\end{equation*}	
	Using this fact and the Parseval's identity we obtain
	\begin{equation}\label{eq:preliminary}
		\int_{\R^d} \frac{|L_r \psi(x)|^2}{|x|^\alpha}\, dx
		=\int_0^\infty \int_{\mathbb{S}^{d-1}} \Big |\sum_{m\in \mathcal{I}} L_rf_m(r)u_m(\omega)\Big |^2 r^{d-\alpha-1}\, dr\, d\omega
		=\sum_{m\in \mathcal{I}}\int_0^\infty |L_rf_m(r)|^2 r^{d-\alpha-1}\, dr.
	\end{equation}
	Let us consider $\int_0^\infty |L_rf_m(r)|^2 r^{d-\alpha-1}\, dr.$ Using the explicit form of $L_r,$ namely $L_r=-\partial_{rr} -\tfrac{d-1}{r}\partial_r,$ and integrating by parts, we obtain 
	\begin{equation*}
		\begin{split}
			\int_0^\infty &|L_rf_m(r)|^2 r^{d-\alpha-1}\, dr\\
			&=\int_0^\infty |f_m''(r)|^2r^{d-\alpha-1}\, dr
			+ (d-1)^2\int_0^\infty |f_m'(r)|^2 r^{d-\alpha-3}\, dr
			+ 2(d-1)\Re \int_0^\infty f_m''(r) \overline{f_m'(r)} r^{d-\alpha-2}\, dr\\
			&=\int_0^\infty |f_m''(r)|^2r^{d-\alpha-1}\, dr
			+ (d-1)^2\int_0^\infty |f_m'(r)|^2 r^{d-\alpha-3}\, dr
			- (d-1)(d-\alpha-2)\int_0^\infty |f_m'(r)|^2 r^{d-\alpha-3}\, dr.
		\end{split}
	\end{equation*}
	Plugging the last identity in~\eqref{eq:preliminary} gives
	\begin{multline}\label{eq:last-r}
			\int_{\R^d} \frac{|L_r \psi(x)|^2}{|x|^\alpha}\, dx
			=\sum_{m\in \mathcal{I}} 
			\left\{ 
				\int_0^\infty |f_m''(r)|^2r^{d-\alpha-1}\, dr
			+ (d-1)^2\int_0^\infty |f_m'(r)|^2 r^{d-\alpha-3}\, dr \right.\\
			\left.- (d-1)(d-\alpha-2)\int_0^\infty |f_m'(r)|^2 r^{d-\alpha-3}\, dr			
			 \right\}.
	\end{multline}
	Again using the decomposition~\eqref{eq:decomposition} one has
	\begin{equation*}
		\Lambda_\omega \psi(x)
		=\sum_{m\in \mathcal{I}} f_m(r) \Lambda_\omega u_m(\omega) 
		=\sum_{m\in \mathcal{I}} \lambda_m f_m(r) u_m(\omega),
	\end{equation*}
	where in the second equality we have used that $\{u_m\}_{m\in \mathcal{I}}$ are eigenfunctions of the operator $\Lambda_\omega$ with corresponding eigenvalues $\lambda_m,$ $m\in \mathcal{I}.$ Using again Parceval's identity one gets
	\begin{equation}\label{eq:last-angular}
		\int_{\R^d} \frac{|\Lambda_\omega \psi(x)|^2}{|x|^{\alpha + 4}}\, dx
		=\int_0^\infty \int_{\mathbb{S}^{d-1}} |\sum_{m\in \mathcal{I}} \lambda_m f_m(r)u_m(\omega)|^2 r^{d-\alpha-5}\, dr\, d\omega
		=\sum_{m\in \mathcal{I}} \lambda_m^2 \int_0^\infty |f_m(r)|^2 r^{d-\alpha-5}\, dr.
	\end{equation}
	Similarly as above, it is easy to check that the following identity holds:
	\begin{equation}\label{eq:last-mixed}
		2\Re \int_{\R^d} \frac{L_r \psi(x) \overline{\Lambda_\omega \psi(x)}}{|x|^{\alpha + 2}}\, dx
		=
		\sum_{m\in \mathcal{I}}\lambda_m \left\{
		2 \int_0^\infty |f_m'(r)|^2 r^{d-\alpha-3}\, dr
		+ (\alpha+ 2)(d-\alpha-4) \int_0^\infty |f_m(r)|^2 r^{d-\alpha-5}\, dr
		 \right\}.
	\end{equation}
	Plugging~\eqref{eq:last-r},~\eqref{eq:last-angular} and~\eqref{eq:last-mixed} in~\eqref{eq:pre-sec} gives~\eqref{eq:second-order}. 

Now, from the definition of $\mathcal{D}$ one has
\begin{equation}
	\label{eq:pre-first}
	\int_{\R^d} \frac{\mathcal{D}(\psi)(x)}{|x|^{\alpha+2}}\, dx=
		\int_{\R^d} \frac{|\partial_r \psi(x)|^2}{|x|^{\alpha+ 2}}\, dx + \int_{\R^d} \frac{|\Lambda_\omega^{1/2} \psi(x)|^2}{|x|^{\alpha + 4}}\, dx.
\end{equation}
Similarly as above, one checks that the following identities hold true:
\begin{equation}\label{eq:last-first2}
		\int_{\R^d} \frac{|\partial_r \psi(x)|^2}{|x|^{\alpha+ 2}}\, dx
		= \sum_{m\in \mathcal{I}} \int_0^\infty |f_m'(r)|^2r^{d-\alpha-3}\, dr,
		\quad \text{and} \quad
		\int_{\R^d} \frac{|\Lambda_\omega^{1/2} \psi(x)|^2}{|x|^{\alpha + 4}}\, dx=
		\sum_{m\in \mathcal{I}} \lambda_m \int_0^\infty |f_m(r)|^2r^{d-\alpha-5}\, dr.
	\end{equation}
Eventually, plugging	~\eqref{eq:last-first2} in~\eqref{eq:pre-first} gives~\eqref{eq:first-order} and, thus, the thesis.
\end{proof}   
\noindent
We are now ready to prove Theorem~\ref{thm:main-general}. 
\begin{proof}[Proof of Theorem~\ref{thm:main-general}]
	The proof is based on the strategy introduced by Cazacu in~\cite{Cazacu2019}. Let us first split~\eqref{eq:second-order} as follows:
	\begin{equation}\label{eq:I+II} 
			\int_{\R^d} \frac{|\mathcal{L}\psi(x)|^2}{|x|^{\alpha}}\, dx= I+ II,
	\end{equation}	
	where
	\begin{multline*}
	I:=
	\sum_{\substack{m\in \mathcal{I}\\\lambda_m\neq 0}} 
			\left \{ 
			\int_0^\infty |f_m''(r)|^2 r^{d-\alpha-1}\, dr
			+ [(d-1)(\alpha+1)]\int_0^\infty |f_m'(r)|^2 r^{d-\alpha-3}\, dr \right\}\\
			+ \sum_{\substack{m\in \mathcal{I}\\\lambda_m\neq 0}} \lambda_m
			\left \{  
			2\int_0^\infty |f_m'(r)|^2r^{d-\alpha-3}\, dr + [(\alpha+2)(d-\alpha-4) +\lambda_m]\int_0^\infty |f_m(r)|^2\, r^{d-\alpha-5}\, dr
			\right \}
	\end{multline*}
	and
	\begin{equation*}
	II:=\sum_{\substack{m\in \mathcal{I}\\\lambda_m=0}}
			\left\{
			\int_0^\infty |f_m''(r)|^2 r^{d-\alpha-1}
			+ [(d-1)(\alpha +1)]\int_0^\infty |f_m'(r)|^2 r^{d-\alpha-3}\, dr		
			\right\}.
	\end{equation*}
	We estimate $II$ first. 
	Using the 1D weighted Hardy inequality
	\begin{equation}\label{eq:1d-Hardy}
	\int_0^\infty |f'(r)|^{2} r^{t+2}\, dr\geq \left(\frac{t+1}{2} \right)^2 \int_0^\infty |f(r)|^2 r^t\, dr, \qquad t\in \R,
\end{equation}
which is valid for any distribution $f$ on $(0,\infty)$ such that the integral on the left hand side of~\eqref{eq:1d-Hardy} is finite
(see \emph{e.g.}~\cite[Prop.2.4]{CP2018}), we have
	\begin{equation}\label{eq:II}
		\begin{split}
			II&\geq \sum_{\substack{m\in \mathcal{I}\\\lambda_m=0}}
			\frac{(d+\alpha)^2}{4}\int_0^\infty |f_m'(r)|^2 r^{d-\alpha-3}\, dr\\
			&=\sum_{\substack{m\in \mathcal{I}\\\lambda_m=0}}
			\frac{(d+\alpha)^2}{4}
			\left[ \int_0^\infty |f_m'(r)|^2 r^{d-\alpha-3}\, dr
			+ \lambda_m \int_0^\infty |f_m(r)|^2 r^{d-\alpha-5}\, dr	
			\right].
		\end{split}
	\end{equation}
	Let $\varepsilon \in \R$ to be fixed later (in particular, the forthcoming choice of $\varepsilon$ will satisfy $\varepsilon/\lambda_m+2\geq0$). We split $I$ as $I=I_{1, \varepsilon}$ + $I_{2,\varepsilon},$ where
	\begin{equation*}
		I_{1,\varepsilon}:=\sum_{\substack{m\in \mathcal{I}\\\lambda_m\neq 0}}
		\left\{
		\int_0^\infty |f_m''(r)|^2 r^{d-\alpha-1}\, dr + [(d-1)(\alpha + 1) -\varepsilon] \int_0^\infty |f_m'(r)|^2 r^{d-\alpha-3}\, dr
		\right\},
	\end{equation*}
	and 
	\begin{equation*}
	I_{2,\varepsilon}=\sum_{\substack{m\in \mathcal{I}\\ \lambda_m\neq 0}} \lambda_m
	\left[
 \left(\frac{\varepsilon}{\lambda_m} + 2 \right) 
 \int_0^\infty |f_m'(r)|^2 r^{d-\alpha-3}\, dr
 + [(\alpha+2)(d-\alpha-4) + \lambda_m] \int_0^\infty |f_m(r)|^2 r^{d-\alpha-5}\, dr	
 \right].
 \footnote{As we will see below, our choice of the parameter $\varepsilon$ will depend on $m,$ therefore the notation $I_{1,\varepsilon}$ and $I_{1,\varepsilon}$ used for these two sums is not entirely correct. A better choice would have been to consider the splitting $I=\sum_{\substack{m\in \mathcal{I}\\ \lambda_m\neq 0}}(I_{1, \varepsilon_m} + I_{2, \varepsilon_m}),$ with $I_{1,\varepsilon_m}$ and $I_{2,\varepsilon_m}$ being the terms inside the sums over $m.$ Anyway we decided to avoid it not to weight down the notation.}
	\end{equation*}
	By~\eqref{eq:1d-Hardy}, we get 
	\begin{equation}\label{eq:last-first}
		I_{1,\varepsilon}\geq \sum_{\substack{m\in \mathcal{I}\\ \lambda_m\neq 0}}\left[\frac{(d+\alpha)^2}{4} - \varepsilon \right]
		\int_0^\infty |f_m'(r)|^2 r^{d-\alpha-3}\, dr,
	\end{equation}
	\begin{equation}\label{eq:last-second}
		I_{2,\varepsilon}\geq \sum_{\substack{m\in \mathcal{I}\\ \lambda_m\neq 0}} \lambda_m \left[\frac{\varepsilon}{\lambda_m}\frac{(d-\alpha-4)^2}{4} + \frac{(d-\alpha-4)(d+\alpha)}{2} + \lambda_m\right]\int_0^\infty |f_m(r)|^2r^{d-\alpha-5}\, dr.
	\end{equation}
	Let $\varepsilon>0$ be chosen such that
	\begin{equation*}
		\frac{(d+\alpha)^2}{4}-\varepsilon
		= \frac{\varepsilon}{\lambda_m}\frac{(d-\alpha-4)^2}{4} + \frac{(d-\alpha-4)(d+\alpha)}{2} + \lambda_m,
	\end{equation*}
	which yields
	\begin{equation*}
		\varepsilon(d)=\frac{\lambda_m[(d+\alpha)(-d+3\alpha+8)-4\lambda_m]}{4\lambda_m +(d-\alpha-4)^2}.
	\end{equation*}
	We stress that with this choice of $\varepsilon$ one has $\tfrac{\varepsilon}{\lambda_m} + 2\geq 0$ (this in particular justifies the possibility to apply~\eqref{eq:1d-Hardy} in $(\tfrac{\varepsilon}{\lambda_m} + 2)\int_0^\infty |f_m'(r)|^2r^{d-\alpha-3}\, dr$ above). Indeed 
	\begin{equation*}
		\frac{\varepsilon}{\lambda_m} + 2= 1 + \frac{4(\alpha+ 2)^2}{4\lambda_m + (d-\alpha-4)^2}\geq 0.
	\end{equation*}	
	In addition, one has
	\begin{equation}\label{eq:I}
		I=I_{1,\varepsilon} + I_{2,\varepsilon}
		\geq\sum_{\substack{m\in \mathcal{I}\\ \lambda_m\neq 0}}
		\frac{(4\lambda_m + (d+\alpha)(d-\alpha-4))^2}{4(4\lambda_m +(d-\alpha-4)^2)} 
		\left \{
		\int_0^\infty |f_m'(r)|^2 r^{d-\alpha-3}\, dr
		+ \lambda_m \int_0^\infty |f_m(r)|^2r^{d-\alpha-5}\, dr
		\right \}.
	\end{equation}
	Plugging estimates~\eqref{eq:I} and~\eqref{eq:II} in~\eqref{eq:I+II}
	we have
	\begin{equation}\label{eq:final}
		\begin{split}
		&\int_{\R^d} \frac{|\mathcal{L}\psi(x)|^2}{|x|^\alpha}\, dx\\		
		&\geq \min \Bigg (
		\tfrac{(d+\alpha)^2}{4};
		\min_{\substack{m\in \mathcal{I}\\ \lambda_m\neq 0}}
		\tfrac{(4\lambda_m + (d+\alpha)(d-\alpha-4))^2}{4(4\lambda_m +(d-\alpha-4)^2)}
		\Bigg)
		\sum_{m\in \mathcal{I}}
		\left \{
		\int_0^\infty |f_m'(r)|^2 r^{d-\alpha-3}\, dr
		+ \lambda_m \int_0^\infty |f_m(r)|^2r^{d-\alpha-5}\, dr \right \}
    	\\
    	&= \min \Bigg (
		\tfrac{(d+\alpha)^2}{4};
		\min_{\substack{m\in \mathcal{I}\\ \lambda_m\neq 0}}
		\tfrac{(4\lambda_m + (d+\alpha)(d-\alpha-4))^2}{4(4\lambda_m +(d-\alpha-4)^2)}
		\Bigg)
		% &=\min_{m\in \mathcal{I}} \frac{(4\lambda_m + (d+\alpha)(d-\alpha-4))^2}{4(4\lambda_m +(d-\alpha-4)^2)}
		\int_{\R^d} \frac{\mathcal{D}(\psi)(x)}{|x|^{\alpha+ 2}}\, dx, 
		\end{split}
	\end{equation}
	where in the last identity we have used~\eqref{eq:first-order}.
	Notice that if $d-\alpha-4\neq 0,$ then we have
		\begin{equation*}
		\frac{(4\lambda_m + (d+\alpha)(d-\alpha-4))^2}{4(4\lambda_m +(d-\alpha-4)^2)}=\frac{(d+\alpha)^2}{4},
		\qquad \text{if } \lambda_m=0.
	\end{equation*}
	This allows us to write the minimum in~\eqref{eq:final} in a more compact form, thus~\eqref{eq:final} can be rewritten as
	\begin{equation*}
	\int_{\R^d} \frac{|\mathcal{L}\psi(x)|^2}{|x|^{\alpha}}\, dx
	\geq 
	\min_{m\in \mathcal{I}} \frac{(4\lambda_m + (d+\alpha)(d-\alpha-4))^2}{4(4\lambda_m +(d-\alpha-4)^2)}
	\int_{\R^d} \frac{\mathcal{D}(\psi)(x)}{|x|^{\alpha+ 2}}\, dx.
	\end{equation*}
	On the other hand, if $d-\alpha-4=0$ the minimum in~\eqref{eq:final} becomes $\min \big((d-2)^2; \min_{\substack{m\in \mathcal{I}\\\lambda_m\neq 0}} \lambda_m\big ).$
	 This concludes the proof.
\end{proof}

\medskip
\noindent
We now pass to the proof of the Hardy-type inequality contained in Theorem~\ref{thm:Hardy-cdc}.
\begin{proof}[Proof of Theorem~\ref{thm:Hardy-cdc}]
From~\eqref{eq:first-order} (replacing $\alpha + 2$ with $\beta$) one has
\begin{equation*}
\int_{\R^d} \frac{\mathcal{D}(\psi)(x)}{|x|^{\beta}}\, dx=
	\sum_{m\in \mathcal{I}}
	\left \{\int_0^\infty |f_m'(r)|^2r^{d-\beta-1}\, dr 
	+ \lambda_m \int_0^\infty |f_m(r)|^2 r^{d-\beta-3} \,dr
	\right \}.
\end{equation*}
Using in the first integral of the right hand side of this identity the 1D-weighted Hardy inequality~\eqref{eq:1d-Hardy}, one gets
\begin{equation*}
	\begin{split}
	\int_{\R^d} \frac{\mathcal{D}(\psi)(x)}{|x|^{\beta}}\, dx
	&\geq 
	\min_{m\in \mathcal{I}} \Big \{ \lambda_m + \frac{(d-\beta -2)^2}{4} \Big\} \sum_{m\in \mathcal{I}} \int_0^\infty |f_m(r)|^2 r^{d-\beta-3}\, dr\\
	&=\min_{m\in \mathcal{I}} \Big \{ \lambda_m + \frac{(d-\beta -2)^2}{4} \Big\}\int_{\R^d} \frac{|\psi(x)|^2}{|x|^{\beta +2}}\,dx,
	\end{split}
\end{equation*}
where in the last identity we just used Parceval's identity as in the proof of Theorem~\ref{thm:main-general}. This concludes the proof.
\end{proof}
%
%\medskip
%\noindent
%We now pass to the proof of Theorem~\ref{thm:1D}.
%\begin{proof}[Proof of Theorem~\ref{thm:1D}]
%	Inequality~\eqref{eq:1dHardy-Rellich} is an immediate consequence of the estimate
%\begin{equation}\label{eq:0infty}
%\int_0^\infty \frac{|\psi''(x)|^2}{|x|^\alpha}\, dx\geq \frac{(\alpha+1)^2}{4}\int_0^\infty \frac{|\psi'(x)|^2}{|x|^{\alpha+2}}\, dx,
%\end{equation}
%for any $\psi\in C^\infty_0((0,\infty))$, which follows by~\eqref{eq:1d-Hardy} applied to $\psi'$.
%\end{proof}

\section{Proof of Theorem~\ref{thm:minimizing}. Optimality of $C(d,\alpha).$ }\label{sec:optimality}
Let $\psi_\epsilon$ the sequence defined in~\eqref{eq:minimizing}. 
We consider first the case $d-\alpha-4\neq 0$ or $d-\alpha-4=0$ and $C(d,\alpha)=\lambda_{m_0}.$
To shorten the notation we write  
\begin{equation*}
	\psi_\epsilon(x)= f_\epsilon(r)u_{m_0}(\omega),
\end{equation*}  
where $f_\epsilon(r)$ represents the radial part of $\psi_\epsilon,$ namely $f_\epsilon(r):=r^{-\frac{(d-\alpha-4)}{2}}g_\epsilon(r)$ and $g_\epsilon$ defined in~\eqref{eq:cutoff}.
As in Lemma~\ref{lemma:1d-reduction} one easily has
\begin{equation}\label{eq:lemma-minimizing}
	\begin{split}
	%\label{eq:second-order}
	&
	\begin{multlined}
	\int_{\R^d} \frac{|\mathcal{L}\psi_\epsilon(x)|^2}{|x|^\alpha}\, dx=
	\int_0^\infty |f_\epsilon''(r)|^2r^{d-\alpha-1}\, dr 
	+ [(d-1)(\alpha +1) + 2\lambda_{m_0}]\int_0^\infty |f_\epsilon'(r)|^2 r^{d-\alpha-3} \,dr \\
	+ \lambda_{m_0} [(\alpha + 2)(d-\alpha -4) + \lambda_{m_0}] \int_0^\infty |f_\epsilon(r)|^2 r^{d-\alpha-5} \,dr,
	\end{multlined}
	\\
	%\label{eq:first-order}
	&\int_{\R^d} \frac{\mathcal{D}(\psi_\epsilon)(x)}{|x|^{\alpha+2}}\, dx=
	\int_0^\infty |f_\epsilon'(r)|^2r^{d-\alpha-3}\, dr 
	+ \lambda_{m_0} \int_0^\infty |f_\epsilon(r)|^2 r^{d-\alpha-5} \,dr.	
	\end{split}
\end{equation} 
Differentiating $f_\epsilon$ with respect to $r$ gives 
\begin{equation}\label{eq:der1}
	f_\epsilon'(r)=
	- \frac{(d-\alpha-4)}{2} r^{-\frac{d-\alpha-2}{2}} g_\epsilon(r) 
	+ r^{-\frac{d-\alpha-4}{2}} g_\epsilon'(r),
\end{equation}
and 
\begin{equation}\label{eq:der2}
f_\epsilon''(r)=
\frac{(d-\alpha-4)}{2}\frac{(d-\alpha-2)}{2}r^{-\frac{d-\alpha}{2}} g_\epsilon(r)
 -2\frac{(d-\alpha-4)}{2}r^{-\frac{d-\alpha-2}{2}}g_\epsilon'(r)
+r^{-\frac{d-\alpha-4}{2}}g_\epsilon''(r).
\end{equation}
From the definition of $g_\epsilon$ in~\eqref{eq:cutoff}, the integrals in~\eqref{eq:lemma-minimizing} are supported over the interval $[\epsilon, 1/\epsilon].$ Now we consider separately the contributions of those integrals over the three sub-intervals $[\epsilon,2\epsilon],$ $[2\epsilon, 1/2\epsilon]$ and $[1/2\epsilon,1/\epsilon].$ We will see that the sole $\epsilon$-dependent contribution comes from the integration over $[2\epsilon,1/2\epsilon],$ whereas the integrals over $[\epsilon, 2\epsilon]$ and $[1/2\epsilon,1/\epsilon]$ are $\mathcal{O}(1)$ in the limit $\epsilon$ goes to $0.$

We start considering the integrals over $[\epsilon,2\epsilon].$ Using the explicit expressions for $f_\epsilon'$ and $f_\epsilon''$ in~\eqref{eq:der1} and~\eqref{eq:der2} respectively, one has
\begin{multline}\label{eq:int1}
	\int_\epsilon^{2\epsilon} |f_\epsilon''(r)|^2r^{d-\alpha-1}\, dr\\
	\leq
	3\Bigg\{
	\frac{(d-\alpha-4)^2}{4} \frac{(d-\alpha-2)^2}{4}\int_\epsilon^{2\epsilon} r^{-1}g_\epsilon^2(r)\,dr
	+4\frac{(d-\alpha-4)^2}{4}\int_\epsilon^{2\epsilon} r g_\epsilon'^{\,2}(r)\, dr
	+ \int_\epsilon^{2\epsilon} r^3g_\epsilon''^{\,2}(r)\, dr
	\Bigg\}.
\end{multline}
Now, using again the property of the function $g_\epsilon,$ it is easy to see that 
\begin{equation}\label{eq:no-dep}
	\begin{split}
	&\int_\epsilon^{2\epsilon} r^{-1}g_\epsilon^2(r)\, dr
	\leq \int_\epsilon^{2\epsilon} r^{-1}\, dr=\ln(2);\\
	&\int_\epsilon^{2\epsilon} r g_\epsilon'^{\,2}(r)\, dr
	\leq 2\epsilon \Big(\frac{c}{\epsilon}\Big)^2\epsilon=2c^2;\\
	&\int_\epsilon^{2\epsilon} r^3g_\epsilon''^{\,2}(r)\, dr
	\leq (2\epsilon)^3\Big(\frac{c}{\epsilon^2}\Big)^2\epsilon=8c^2.
\end{split}
\end{equation}
In particular, the three integrals above do not depend on $\epsilon,$ therefore from~\eqref{eq:int1} we have 
\begin{equation*}
	\int_\epsilon^{2\epsilon} |f_\epsilon''(r)|^2r^{d-\alpha-1}\, dr=\mathcal{O}(1).
\end{equation*}
Similarly, one has
\begin{equation*}
	\int_\epsilon^{2\epsilon} |f_\epsilon'(r)|^2r^{d-\alpha-3}\, dr
	\leq 
	2\Bigg\{
	\frac{(d-\alpha-4)^2}{4} \int_\epsilon^{2\epsilon} r^{-1} g_\epsilon^2(r)\, dr
	+\int_\epsilon^{2\epsilon}	r g_\epsilon'^{\,2}(r)\,dr
	\Bigg\},
\end{equation*} 
and from~\eqref{eq:no-dep}
\begin{equation*}
	\int_\epsilon^{2\epsilon} |f_\epsilon'(r)|^2r^{d-\alpha-3}\, dr=\mathcal{O}(1).
\end{equation*}
Analogously,
\begin{equation*}
	\int_\epsilon^{2\epsilon} |f_\epsilon(r)|^2r^{d-\alpha-5}\, dr
	=\int_\epsilon^{2\epsilon} r^{-1}g_\epsilon^2(r)\, dr=\mathcal{O}(1).
\end{equation*}
To sum up, one has
\begin{equation}\label{eq:O(1)}
	\begin{split}
		\int_\epsilon^{2\epsilon} |f_\epsilon''(r)|^2r^{d-\alpha-1}\, dr
	&=\mathcal{O}(1),\\
	\int_\epsilon^{2\epsilon} |f_\epsilon'(r)|^2r^{d-\alpha-3}\, dr
	&=\mathcal{O}(1),\\
	\int_\epsilon^{2\epsilon} |f_\epsilon(r)|^2r^{d-\alpha-5}\, dr
	&=\mathcal{O}(1).
	\end{split}
\end{equation}
When we are on $[2\epsilon, 1/2\epsilon],$ then $g_\epsilon=1$ and $f_\epsilon, f_\epsilon'$ and $f_\epsilon''$ assume the particularly simple form
\begin{equation*}
f_\epsilon(r)=r^{-\frac{d-\alpha-4}{2}};
\qquad
f_\epsilon'(r)=-\frac{(d-\alpha-4)}{2}r^{-\frac{d-\alpha-2}{2}};
\qquad
f_\epsilon''(r)=\frac{(d-\alpha-4)}{2}\frac{(d-\alpha-2)}{2}r^{-\frac{d-\alpha}{2}}.
\end{equation*}
Now, a direct computation gives
\begin{equation}\label{eq:non-trivial}
	\begin{split}
	\int_{2\epsilon}^{1/2\epsilon} |f_\epsilon''(r)|^2r^{d-\alpha-1}\, dr
	&=-\frac{(d-\alpha-4)^2}{4}\frac{(d-\alpha-2)^2}{4} \ln(4\epsilon^2),\\
	\int_{2\epsilon}^{1/2\epsilon} |f_\epsilon'(r)|^2r^{d-\alpha-3}\, dr
	&=-\frac{(d-\alpha-4)^2}{4}\ln(4\epsilon^2),\\
	\int_{2\epsilon}^{1/2\epsilon} |f_\epsilon(r)|^2r^{d-\alpha-5}\, dr
	&=-\ln(4\epsilon^2).
	\end{split}
\end{equation}
In the interval $[1/2\epsilon,1/\epsilon]$ analogous computations as the ones in $[\epsilon, 2\epsilon]$ give
\begin{equation}\label{eq:O(1)bis}
	\begin{split}
		\int_{1/2\epsilon}^{1/\epsilon} |f_\epsilon''(r)|^2r^{d-\alpha-1}\, dr
	&=\mathcal{O}(1),\\
	\int_{1/2\epsilon}^{1/\epsilon} |f_\epsilon'(r)|^2r^{d-\alpha-3}\, dr
	&=\mathcal{O}(1),\\
	\int_{1/2\epsilon}^{1/\epsilon} |f_\epsilon(r)|^2r^{d-\alpha-5}\, dr
	&=\mathcal{O}(1).
	\end{split}
\end{equation}
Using~\eqref{eq:O(1)},~\eqref{eq:non-trivial} and~\eqref{eq:O(1)bis} in~\eqref{eq:lemma-minimizing} we have
\begin{equation*}
	\begin{split}
\frac{\int_{\R^d} |\mathcal{L}\psi_\epsilon(x)|^2/|x|^\alpha\,dx}{\int_{\R^d} \mathcal{D}(\psi_\epsilon)(x)/|x|^{\alpha+2}\, dx}
&=
\frac{\int_{\R^d\cap \{2\epsilon\leq |x|\leq 1/2\epsilon\}} |\mathcal{L}\psi_\epsilon(x)|^2/|x|^\alpha\, dx + \mathcal{O}(1) }{\int_{\R^d\cap \{2\epsilon\leq |x|\leq 1/2\epsilon\}} |\mathcal{L}^{1/2}\psi_\epsilon(x)|^2/|x|^{\alpha+2}\, dx  + \mathcal{O}(1)}\\
%&=\frac{-\ln(2\epsilon)
%\Big\{
%\frac{(d-\alpha-4)^2}{4}\frac{(d-\alpha-2)^2}{4} 
%+ [(d-1)(\alpha+1) + 2\lambda_{m_0}]\frac{(d-\alpha-4)^2}{4}
%+\lambda_{m_0}[(\alpha+2)(d-\alpha-4) + \lambda_{m_0}] + \mathcal{O}(1/ln(2\epsilon))
%\Big\}}
%{-\ln(2\epsilon)
%\Big\{
%\frac{(d-\alpha-4)^2}{4} + \lambda_{m_0} + \mathcal{O}(1/\ln(R/2\epsilon))
%\Big\}}
&=\frac{
\big[ (d-\alpha-4)(d+\alpha) +4\lambda_{m_0}\big]^2
+\mathcal{O}(1/\ln(4\epsilon^2))
}
{
4[(d-\alpha-4)^2+ 4\lambda_{m_0}]
+\mathcal{O}(1/\ln(4\epsilon^2))
}\vspace{0.1cm}\\
&\searrow C(d,\alpha), \qquad \text{as }\epsilon\searrow 0.
	\end{split}
\end{equation*}
Now we consider the case $d-\alpha-4=0$ and $C(d,\alpha)=(d-2)^2.$ In this case $\psi_\epsilon(x):=h_\epsilon(r),$ \emph{i.e.}  $\psi_\epsilon(x)$ is radial. Since the spherical part is missing, performing analogous computations as in Lemma~\ref{lemma:1d-reduction} one gets
\begin{equation*}
	\begin{split}
	\int_{\R^d} \frac{|\mathcal{L}_\epsilon\psi(x)|^2}{|x|^\alpha}\, dx
	&=
	|\mathbb{S}^{d-1}|\bigg(\int_0^\infty |h_\epsilon''(r)|^2r^3\, dr 
	+ [(d-1)(d-3)]\int_0^\infty |h_\epsilon'(r)|^2r\, dr\bigg)\\
	\int_{\R^d} \frac{\mathcal{D}(\psi_\epsilon)(x)}{|x|^{\alpha+2}}\, dx&
	=|\mathbb{S}^{d-1}|\int_0^\infty |h_\epsilon'(r)|^2r\, dr.	
	\end{split}
\end{equation*}
From the definition~\eqref{eq:h_eps} of $h_\epsilon$ one has
\begin{equation*}
	h_\epsilon'(r)=r^{-1}g_\epsilon(r), 
	\qquad 
	h_\epsilon''(r)=-r^{-2}g_\epsilon(r) + r^{-1}g_\epsilon'(r).
\end{equation*}
As above we consider separately the integrals over the sub-interval $[\epsilon,2\epsilon], [2\epsilon,1/2\epsilon]$ and $[1/2\epsilon,1/\epsilon].$

In $[\epsilon,2\epsilon]$ one has
\begin{equation*}
\int_\epsilon^{2\epsilon} |h_\epsilon''|^2r^3\,dr\leq 2\Bigg(\int_\epsilon^{2\epsilon}r^{-1}g_\epsilon^2(r)\, dr + \int_\epsilon^{2\epsilon} r g_\epsilon'^{\,2}(r)\, dr\Bigg)
\end{equation*}
and 
\begin{equation*}
	\int_\epsilon^{2\epsilon} |h_\epsilon'(r)|^2r\,dr
	=\int_\epsilon^{2\epsilon} r^{-1}g_\epsilon^2(r)\, dr.	
\end{equation*}
Using~\eqref{eq:no-dep} one has
\begin{equation*}
\begin{split}
	\int_\epsilon^{2\epsilon} |h_\epsilon''|^2r^3\,dr&=\mathcal{O}(1),\\
	\int_\epsilon^{2\epsilon} |h_\epsilon'(r)|^2r\,dr&=\mathcal{O}(1).
\end{split}
\end{equation*}
We now consider the integrals over $[2\epsilon, 1/2\epsilon].$ Here $h_\epsilon'(r)=r^{-1}$ and $h_\epsilon''(r)=-r^{-2}.$ Thus
\begin{equation*}
	\int_{2\epsilon}^{1/2\epsilon} |h_\epsilon''|^2r^3\,dr
	=\int_{2\epsilon}^{1/2\epsilon} |h_\epsilon'|^2r\,dr
	=-\ln(4\epsilon^2).
\end{equation*}
Finally the integrals over $[1/2\epsilon,1/\epsilon]$ can be treated similarly to the ones over $[\epsilon, 2\epsilon].$ This gives
\begin{equation*}
	\begin{split}
	\int_{1/2\epsilon}^{1/\epsilon} |h_\epsilon''|^2r^3\,dr&=\mathcal{O}(1),\\
	\int_{1/2\epsilon}^{1/\epsilon} |h_\epsilon'(r)|^2r\,dr&=\mathcal{O}(1).
\end{split}
\end{equation*}
These facts together give
\begin{equation*}
	\begin{split}
\frac{\int_{\R^d} |\mathcal{L}\psi_\epsilon(x)|^2/|x|^\alpha\,dx}{\int_{\R^d} \mathcal{D}(\psi_\epsilon)(x)/|x|^{\alpha+2}\, dx}
&=
\frac{\int_{\R^d\cap \{2\epsilon\leq |x|\leq 1/2\epsilon\}} |\mathcal{L}\psi_\epsilon(x)|^2/|x|^\alpha\, dx + \mathcal{O}(1) }{\int_{\R^d\cap \{2\epsilon\leq |x|\leq 1/2\epsilon\}} \mathcal{D}(\psi_\epsilon)(x)/|x|^{\alpha+2}\, dx  + \mathcal{O}(1)}\\
&=\frac{
(d-2)^2
+\mathcal{O}(1/\ln(4\epsilon^2))
}
{
1
+\mathcal{O}(1/\ln(4\epsilon^2))
}\vspace{0.1cm}\\
&\searrow (d-2)^2, \qquad \text{as }\epsilon\searrow 0.
	\end{split}
\end{equation*}
This concludes the proof of the optimality of $C(d,\alpha).$

\medskip
\noindent
It remains to show that the constant $C(d,\alpha)$ is not attained. This fact is a consequence of the non-attainability of the best constant in the 1D-Hardy inequality~\eqref{eq:1d-Hardy}. Indeed, going back through the proof of Theorem~\ref{thm:main-general}, one realizes that for $C(d,\alpha)$ to be attained, it is necessary to have equality in the estimates where we applied~\eqref{eq:1d-Hardy}. More precisely, we want to have equality in
\begin{equation}\label{eq:first-att} 
	\int_0^\infty |f_m''(r)|^2 r^{d-\alpha-1}\, dr
	\geq 
	\Big(
	\frac{d-\alpha-2}{2}
	\Big)^2
	\int_0^\infty |f_m'(r)|^2 r^{d-\alpha-3}\, dr,
\end{equation}  
or, equivalently, in
\begin{equation}\label{eq:last-att} 
	\int_0^\infty |f_m'(r)|^2 r^{d-\alpha-3}\, dr
	\geq 
	\Big(
	\frac{d-\alpha-4}{2}
	\Big)^2
	\int_0^\infty |f_m(r)|^2 r^{d-\alpha-5}\, dr.
\end{equation}  
Notice that~\eqref{eq:last-att} is also a consequence of the identity
\begin{equation*}
	\int_0^\infty |f_m'(r)|^2 r^{d-\alpha-3}\, dr
	-
	\Big(
	\frac{d-\alpha-4}{2}
	\Big)^2
	\int_0^\infty |f_m(r)|^2 r^{d-\alpha-5}\, dr
	=\int_0^\infty \Big|\big(r^\frac{d-\alpha-4}{2}f_m(r)\big)'\Big|^2 r\, dr.
\end{equation*} 
In view of the last identity, equality in~\eqref{eq:last-att} is achieved if 
\begin{equation*}
	\big(r^\frac{d-\alpha-4}{2}f_m(r)\big)'=0,
\end{equation*}
which leads to the family of solutions
\begin{equation*}
	f_m(r)=a_m r^{-\frac{d-\alpha-4}{2}} + b_m,
\end{equation*}
for some real constants $a_m, b_m.$ Thus, the fundamental system of solutions is given by $\{r^{-\frac{d-\alpha-4}{2}},1\}.$ Notice that $f_m(r)=1$ is not possible since constant functions are not admissible for inequality~\eqref{eq:last-att}. Moreover, $f_m(r)=r^{-\frac{d-\alpha-4}{2}}$ is not admissible because none of the terms in~\eqref{eq:first-att} are integrable. Thus, we conclude that $C(d,\alpha)$ is not attained.
\qed

\section{Proof of the particular cases: Theorem~\ref{thm:general-electric}, Theorem~\ref{thm:HB-anydimension} and Theorem~\ref{thm:monopole}}\label{sec:consequences}
In order to prove Theorem~\ref{thm:general-electric}, Theorem~\ref{thm:HB-anydimension} and Theorem~\ref{thm:monopole} one simply has to show that the corresponding operators can be recast into the form of the general operator $\mathcal{L}$ defined in~\eqref{eq:operator}. 
\begin{proof}[Proof of Theorem~\ref{thm:general-electric}]
Consider the operator $-\Delta_{a(\theta)}:=-\Delta + \tfrac{a(\theta)}{|x|^2},$ since the function $a=a(\theta)$ depends only on the spherical variable $\theta,$ it is easy to see that $-\Delta_{a(\theta)}$ can be written more conveniently as
\begin{equation*}
	-\Delta_{a(\theta)}=L_r + \frac{1}{r^2} (-\Delta_{\mathbb{S}^{d-1}} + a(\theta)),
	\qquad L_r=-\frac{\partial^2}{\partial r^2}- \frac{d-1}{r}\frac{\partial}{\partial r},
\end{equation*}
thus the operator $\Lambda_\omega$ in~\eqref{eq:operator} is represented by the non-negative, self-adjoint operator $-\Delta_{\mathbb{S}^{d-1}} + a(\theta)$  in $L^2(\mathbb{S}^{d-1};d\theta).$ This operator has been largely studied (see \emph{e.g.}~\cite{FMT2007,FFT2011, FFFP2}). In particular in~\cite[Lemma 2.1]{FMT2007} it has been proved that $\Lambda_\omega=-\Delta_{\mathbb{S}^{d-1}} + a(\theta)$ on $\mathbb{S}^{d-1}$ admits a divergent sequence of eigenvalues $\mu_k,$ $k\in \N_0,$ with finite multiplicity, the first of which satisfies $\mu_0\geq \ess \inf_{\mathbb{S}^{d-1}} a.$ Therefore, the hypotheses of Theorem~\ref{thm:main-general} are satisfied. Thus Theorem~\ref{thm:general-electric} follows from identity~\eqref{eq:cdc-corr} as soon as one checks that the Carré du Champ in this case is given by $\Gamma(\psi)= |\nabla \psi|^2 + \frac{a}{2}|\psi|^2/|x|^2$ and eventually noticing that $\Lambda_\omega |x|^\beta=a|x|^\beta.$ 
\end{proof}

\begin{proof}[Proof of Theorem~\ref{thm:HB-anydimension}]
As in the previous case we show that the Aharonov-Bohm magnetic Laplacian $-\Delta_{A}$ can be written in the form~\eqref{eq:operator} in any dimension $d\geq 2.$ 

For $d\geq 2,$ we take the transformation from Cartesian to spherical coordinates, namely $x=(x_1,x_2, \dots, x_d)\in \R^d$ to $(r, \theta_1, \dots, \theta_{d-1})\in (0,\infty)\times \mathbb{S}^{d-1},$ where $\mathbb{S}^{d-1}$ is the $d-1$-dimensional sphere with respect to the Hausdorff measure in $\R^d,$ given by
\begin{equation*}
	\begin{split}
		&x_1=r\cos\theta_1,\\
		&x_j=r\cos \theta_j \prod_{k=1}^{j-1}\sin \theta_k, \qquad j\in \{2,3\dots, d-1\},\\
		&x_d=r\prod_{k=1}^{d-1}\sin \theta_k.
	\end{split}
\end{equation*}
The corresponding orthogonal unit vectors are given by
\begin{equation*}
	\begin{split}
		&e_r:=(\cos \theta_1, \cos \theta_2\sin \theta_1, \dots, \cos \theta_{d-1}\prod_{k=1}^{d-2}\sin \theta_k, \prod_{k=1}^{d-1}\sin \theta_k),\\
		&\begin{multlined}
			e_{\theta_j}:=(\underbrace{0, \dots, 0}_{j-1}, -\sin \theta_j, \cos \theta_{j+1} \cos \theta_j, \cos\theta_{j+2}\cos \theta_j \sin \theta_{j+1},
		 \dots,\\ 
		 \cos\theta_{d-1}\cos \theta_j \prod_{k=1, k\neq j}^{d-2}\sin \theta_k, \cos \theta_j \prod_{k=1, k\neq j}^{d-1}\sin \theta_k),
	\qquad \qquad j\in \{1, \dots, d-2\},
	\end{multlined}	
	\\
		&e_{\theta_{d-1}}:=(\underbrace{0, \dots, 0}_{d-2},-\sin \theta_{d-1}, \cos \theta_{d-1}).
	\end{split}
\end{equation*}
%It can be shown that the gradient in spherical coordinates is
%\begin{equation*}
%	\nabla=e_r \frac{\partial}{\partial r} + \frac{1}{r}e_{\theta_1} \frac{\partial}{\partial \theta_1} 
%	+\sum_{j=2}^{d-1} \frac{1}{r \prod_{k=1}^{j-1}\sin \theta_k} e_{\theta_j} \frac{\partial}{\partial \theta_j}. 
%\end{equation*}
%The Aharonov-Bohm type vector potential $A$ is customarily defined to be the vector field
%\begin{equation}\label{eq:A-polar-coordinates}
%A:=
%\begin{cases}
%\frac{1}{r}\Psi(\theta_1)e_{\theta_1}, \qquad &\text{if } d=2,\\
%\frac{1}{r\prod_{k=1}^{d-2} \sin \theta_k} \Psi(\theta_{d-1})e_{\theta_{d-1}}, \qquad &\text{if } d\geq3,  
%\end{cases}
%\end{equation}  
%$\Psi\in L^\infty(\mathbb{S}^1),$ defined on $\R^d\setminus \mathcal{G}_d$ where 
%\begin{equation*}
%	\mathcal{G}_d:=
%	\begin{cases}
%	\{0\}, \qquad &\text{if }d=2,\\
%	\{x\in \R^d, x=(r, \theta_1,\dots, \theta_{d-1})\colon r\prod_{k=1}^{d-2}\sin \theta_k=0 \}, \qquad & \text{if }d\geq 3.
%	\end{cases}
%\end{equation*}
Without loss of generality we can assume that the function $\Psi=\Psi(\theta_{d-1})$ is constant, indeed $A$ as defined in~\eqref{eq:A-polar-coordinates} is gauge equivalent to the vector potential $\widetilde{A}$ defined as 

Using spherical coordinates, the Aharonov-Bohm vector potential $A$ defined in~\eqref{eq:AB-gen} can be rewritten as
\begin{equation}\label{eq:A-polar-coordinates}
	A:=
\begin{cases}
\frac{1}{r}\widetilde{\Psi} e_{\theta_1}, \qquad &\text{if } d=2,\\
\frac{1}{r\prod_{k=1}^{d-2} \sin \theta_k} \widetilde{\Psi} e_{\theta_{d-1}}, \qquad &\text{if } d\geq3,  
\end{cases}
\qquad \widetilde{\Psi}:=\frac{1}{2\pi}\int_0^{2\pi}\Psi(\theta)\, d\theta,
\end{equation}
(see~\cite[Section 5.4.2]{BEL} for more details). 
Recalling the following expression for the gradient in spherical coordinates
\begin{equation*}
	\nabla=e_r \frac{\partial}{\partial r} + \frac{1}{r}e_{\theta_1} \frac{\partial}{\partial \theta_1} 
	+\sum_{j=2}^{d-1} \frac{1}{r \prod_{k=1}^{j-1}\sin \theta_k} e_{\theta_j} \frac{\partial}{\partial \theta_j},
\end{equation*}
one checks easily that the magnetic gradient $\nabla_A:=\nabla -iA$ associated to the Aharonov-Bohm magnetic vector potential~\eqref{eq:A-polar-coordinates} can be written as
\begin{equation*}
	\nabla_A=e_r\frac{\partial}{\partial_r} + \frac{1}{r}\nabla_{d,\theta},
	\qquad \nabla_{d,\theta}=e_{\theta_1}\frac{\partial}{\partial\theta_1} + \sum_{j=2}^{d-2} \frac{1}{\prod_{k=1}^{j-1}\sin \theta_k} e_{\theta_j} \frac{\partial}{\partial \theta_j}
	+ \frac{1}{\prod_{k=1}^{d-2}\sin \theta_k}e_{\theta_{d-1}}  \Big(\frac{\partial}{\partial \theta_{d-1}}-i\widetilde{\Psi}\Big). 
\end{equation*}
The corresponding magnetic Laplacian $-\Delta_A:=-\nabla_{\!A}^2$ has the form
\begin{equation*}
	-\Delta_A:=L_r
	+ \frac{1}{r^2}\Lambda_{d,\theta},
\end{equation*}
where
\begin{equation*}
	L_r=-\frac{\partial^2}{\partial r^2}-\frac{d-1}{r} \frac{\partial}{\partial r}
	\qquad
	\text{and}
	\qquad
	\Lambda_{d,\theta}=
	 - \sum_{j=1}^{d-2} \frac{1}{q_j}\Big[(d-j-1)\cot \theta_j \frac{\partial}{\partial \theta_j} + \frac{\partial^2}{\partial \theta_j^2}\Big] + \frac{1}{q_{d-1}}\left(i \frac{\partial}{\partial \theta_{d-1}} + \Psi(\theta_{d-1}) \right)^2,
\end{equation*}
with
\begin{equation*}
	q_j:=
	\begin{system}
	&1, \qquad & \text{if } j=1,\\
	&\prod_{k=1}^{j-1}\sin^2 \theta_k, \qquad & \text{if } j\geq 2. 
	\end{system}
\end{equation*}
An easy computation also shows that the generalisation of the Laplace-Beltrami operator $\Lambda_{d,\theta}$ can be obtained through the angular part of the magnetic gradient as follows $\Lambda_{d,\theta}=-\nabla_{d,\theta}\cdot \nabla_{d,\theta}.$ Moreover, the following identity can be obtained by integration by parts
\begin{equation}\label{eq:grad-LB}
\int_{\mathbb{S}^{d-1}} \psi \Lambda_{d,\theta}\psi\, d\omega=\int_{\mathbb{S}^{d-1}} |\nabla_{d, \theta}\psi|^2\, d\omega.
\end{equation}
Clearly, the operator $\Lambda_{d,\theta}$ plays the role of $\Lambda_\omega$ in~\eqref{eq:operator}. Moreover, in~\cite[Theorem 3.2]{Thomas2007} (see also~\cite{EL2005}) it is proved that the non-negative, self-adjoint magnetic Laplace-Beltrami operator $\Lambda_\omega:=\Lambda_{d,\theta}$ has spectrum consisting of eigenvalues 
\begin{equation*}
	\lambda_m=(m+\widetilde \Psi)(m+\widetilde \Psi + d-2),
\end{equation*}
where $m\in \Z':=\{m\in \Z\colon m\leq 2-d-\widetilde \Psi \, \text{or } m\geq -\widetilde \Psi\}.$ 
 Using now Theorem~\ref{thm:main-general}, one gets Theorem~\ref{thm:HB-anydimension} as soon as it is shown that
 \begin{equation}\label{eq:final-cdc}
 	\int_{\R^d} \frac{\mathcal{D}(\psi)(x)}{|x|^{\alpha+2}}\, dx=\int_{\R^d} \frac{|\nabla_A \psi|^2}{|x|^{\alpha+2}}\, dx.
 \end{equation}   
 By definition 
 \begin{equation*}
 	\int_{\R^d} \frac{\mathcal{D}(\psi)(x)}{|x|^{\alpha+2}}\, dx=
		\int_{\R^d} \frac{|\partial_r \psi(x)|^2}{|x|^{\alpha+ 2}}\, dx + \int_{\R^d} \frac{|\Lambda_{d,\theta}^{1/2} \psi(x)|^2}{|x|^{\alpha + 4}}\, dx,
 \end{equation*}
 where $\Lambda_{d,\theta}^{1/2}$ denotes the square root of the non-negative, self-adjoint magnetic Laplace-Beltrami $\Lambda_{d,\theta}.$  
From identity~\eqref{eq:grad-LB} the following chain of identities holds
\begin{equation*}
	\int_{\mathbb{S}^{d-1}} |\Lambda_{d,\theta}^{1/2}\psi|^2\, d\omega
	=\int_{\mathbb{S}^{d-1}} \overline{\Lambda_{d,\theta}^{1/2}\psi}\Lambda_{d,\theta}^{1/2}\psi\, d\omega
	=\int_{\mathbb{S}^{d-1}} \overline{\psi} \Lambda_{d,\theta}\psi \, d\omega
	=\int_{\mathbb{S}^{d-1}} |\nabla_{d,\theta}\psi|^2\, d\omega,
\end{equation*} 
thus one has 
\begin{equation*}
 	\int_{\R^d} \frac{\mathcal{D}(\psi)(x)}{|x|^{\alpha+2}}\, dx=
		\int_{\R^d} \frac{|\partial_r \psi(x)|^2}{|x|^{\alpha+ 2}}\, dx + \int_{\R^d} \frac{|\nabla_{d,\theta} \psi(x)|^2}{|x|^{\alpha + 4}}\, dx.
 \end{equation*}
Using that
\begin{equation*}
	|\nabla_A \psi|^2=|\partial_r \psi|^2 + \frac{1}{r^2}|\nabla_{d,\theta}\psi|^2,
\end{equation*}
then we have~\eqref{eq:final-cdc} and, in turn, Theorem~\ref{thm:HB-anydimension} is proved.
\end{proof}

\begin{proof}[Proof of Theorem~\ref{thm:monopole}]
The Hamiltonian of a monopole of degree $g$ in $\R^3$ has been intensively studied in
~\cite{CT2010}). In particular it can be shown that the magnetic Laplacian $-\Delta_A$ associated to the vector potential $A$ defined in~\eqref{eq:A-monopole} can be written as
\begin{equation*}
	-\Delta_A=-\frac{\partial^2}{\partial r^2} - \frac{2}{r} \frac{\partial}{\partial r} + \frac{1}{r^2} K_g,
\end{equation*}
where $K_g$ is the angular magnetic Schrödinger operator (see~\cite{CT2010}) and plays the role of $\Lambda_\omega$ in~\eqref{eq:operator}. One can also prove~\cite[Theorem 5.13]{CT2010} that spectrum of $K_g$ is discrete, more precisely it is the sequence $\lambda_k=\frac{1}{4} k(k+2)-g^2,$ $k=2(|g|+l), l\in \N_0.$

Thus, the hypotheses of Theorem~\ref{thm:main-general} are satisfied and therefore Theorem~\ref{thm:monopole} follows.
\end{proof}

%-------------Bibliography---------------%
\providecommand{\bysame}{\leavevmode\hbox to3em{\hrulefill}\thinspace}

%---------------End-------------------%

\end{document}